\documentclass[11 pt]{amsart}
\usepackage{amssymb}
\usepackage{amsmath}
\usepackage{amsthm}
\usepackage{amsfonts}
\usepackage{standalone}

\usepackage[margin=1 in]{geometry}
\usepackage{multicol}
\usepackage{tikz}
\usetikzlibrary{patterns}
\usepackage{lmodern,tikz}
\usetikzlibrary{positioning}
\usepackage{pgfplots,caption}
\pgfplotsset{compat=1.13}
\usepackage[utf8]{inputenc}
\usepackage{graphicx}
\usepackage{verbatim}
\usepackage{enumerate}
\usepackage{mathtools}
\usepackage{xcolor}
\usetikzlibrary{calc}
\usepackage{float}

\usepackage{mathpazo}
\usepackage{euler}
\usepackage{microtype}

\DeclarePairedDelimiter\abs{\lvert}{\rvert}%
\DeclarePairedDelimiter\norm{\lVert}{\rVert}%

\makeatletter
\let\oldabs\abs
\def\abs{\@ifstar{\oldabs}{\oldabs*}}
\let\oldnorm\norm
\def\norm{\@ifstar{\oldnorm}{\oldnorm*}}
\makeatother

\usepackage{amsmath,stackengine,scalerel}
\stackMath
\def\dhat#1{\ThisStyle{\setbox0=\hbox{$\SavedStyle#1$}%
  \stackengine{0pt}{\SavedStyle#1}{\SavedStyle\hspace{.4\ht2}%
  \hat{\vphantom{#1}}\kern\dimexpr3.2\LMpt+2.0pt\relax\hat{\vphantom{#1}}}{O}{c}{F}{T}{L}}%
}

\tikzset{
    triple/.style args={[#1] in [#2] in [#3]}{
        #1,preaction={preaction={draw,#3},draw,#2}
    }
}

\DeclareMathAlphabet{\mathcal}{OMS}{cmsy}{m}{n}
\newcommand{\F}{\mathcal{F}}
\newcommand{\R}{\mathcal{R}}

\newcommand{\M}{\mathcal{M}}
\newcommand{\G}{G^{A,B}}
\newcommand{\T}{\mathcal{T}}

\newcommand{\GG}{\G_{\text{red}}}

\newcommand{\red}{\G_{\text{red}}}
\newcommand{\ol}{\overline}

\usepackage{xpatch}
\makeatletter
\xpatchcmd{\@thm}{\thm@headpunct{.}}{\thm@headpunct{}}{}{}
\makeatother
\newtheorem{lemma}[subsection]{Lemma}
\newtheorem{theorem}[subsection]{Theorem}
\newtheorem{question}[subsection]{Question}
\newtheorem{conjecture}[subsection]{Conjecture}
\theoremstyle{definition}
\newtheorem{definition}[subsection]{Definition}
\newtheorem{corollary}[subsection]{Corollary}
\newtheorem{remark}[subsection]{Remark}
\newtheorem{example}[subsection]{Example}
\newtheoremstyle{case}{}{}{}{}{}{:}{ }{}
\theoremstyle{case}

\title{Biconed Graphs, weighted forests, and $h$-vectors of matroid complexes }
\author[P. Cranford]{Preston Cranford}
\author[A. Dochtermann]{Anton Dochtermann}
\author[E. Haithcock]{Evan Haithcock}
\author[J. Marsh]{Joshua Marsh}
\author[S. Oh]{Suho Oh}
\author[A. Truman]{Anna Truman}

\address{Massachusetts Institute of Technology, Department of Mathematics}
\email{prestonc@mit.edu}

\address{Texas State University, Department of Mathematics}
\email{dochtermann@txstate.edu}

\address{Clemson University, School of Mathematical and Statistical Sciences}
\email{ehaithc@g.clemson.edu}

\address{The University of Texas at Dallas, Department of Mathematical Sciences}
\email{joshuawmarsh@gmail.com}

\address{Texas State University, Department of Mathematics}
\email{s\_o79@txstate.edu}

\address{Grove City College, Department of Mathematics}
\email{actruman@icloud.com}

\date{\today}

\begin{document}

%%%%%%%%%%%%%%%%%%%%%%%%%%%%%%%%%%%%%%%%%%%%%%%%%%%%%%%%%%%%%%%%%%%%
%%%%%%%%%%%%%%%%%%%%%%%%%%%%%%%%%%%%%%%%%%%%%%%%%%%%%%%%%%%%%%%%%%%%

\begin{abstract}
A well-known conjecture of Richard Stanley posits that the $h$-vector of the independence complex of a matroid is a pure ${\mathcal O}$-sequence. The conjecture has been established for various classes but is open for graphic matroids. A biconed graph is a graph with two specified `coning vertices', such that every vertex of the graph is connected to at least one coning vertex. The class of biconed graphs includes coned graphs, Ferrers graphs, and complete multipartite graphs.  We study the $h$-vectors of graphic matroids arising from biconed graphs, providing a combinatorial interpretation of their entries in terms of `$2$-weighted forests' of the underlying graph. This generalizes constructions of Kook and Lee who studied the M\"obius coinvariant (the last nonzero entry of the $h$-vector) of graphic matroids of complete bipartite graphs. We show that allowing for partially $2$-weighted forests gives rise to a pure multicomplex whose face count recovers the $h$-vector, establishing Stanley's conjecture for this class of matroids.  We also discuss how our constructions relate to a combinatorial strengthening of Stanley's Conjecture (due to Klee and Samper) for this class of matroids.

\end{abstract}

\maketitle

%pure $O$-sequence, graphic matroid, $h$-vector, biconed graphs, edge rooted forests, spanning trees, 

\section{Introduction}
%%%%%%%%%%%%%%%%%%%%%%%%%%%%%%%%%%%%%%%%%%%%%%%%%%%%%%%%%%%%%%%%%%%%

A matroid is a combinatorial structure that generalizes various notions of independence that arise in linear algebra, field extensions, graph theory, matching theory, and other areas. A \emph{graphic matroid} ${\mathcal M}(G)$ has its ground set given by the edge set of some finite connected graph $G$, with independent sets given by the sets of edges that do contain a cycle. Given a matroid ${\mathcal M}$, of particular interest is the number of independent sets of ${\mathcal M}$ of a certain size. The $h$-vector of ${\mathcal M}$ encodes this information in a convenient format. The $h$-vector of a matroid provides topological information regarding its underlying simplicial complex and also relates to the notion of \emph{activity} of bases. 

In his work surrounding the Upper Bound Conjecture \cite{Stanley1977}, Stanley proved that if a simplicial complex is Cohen-Macaulay (an algebraic condition on its associated face ring) then its $h$-vector is necessarily an \emph{${\mathcal O}$-sequence}: the entries $h_i$ are given by the number of degree $i$ monomials in some order ideal (see Section \ref{sec:Prelim} for details).  Motivated by these results and the orderly structure of matroids (a type of Cohen-Macaulay simplicial complex), Stanley conjectured \cite{Stanley1977} that the $h$-vectors of matroids satisfy a stronger condition.

\begin{conjecture}
The $h$-vector of a matroid is a pure ${\mathcal O}$-sequence.
\end{conjecture}

Here an ${\mathcal O}$-sequence is \emph{pure} if the maximal elements of the underlying order ideal can be chosen to all have the same degree; again we refer to Section \ref{sec:Prelim} for details. Despite receiving considerable attention for over four decades, Stanley's conjecture remains mostly wide open today.  It has been established for some specific classes of matroids, in particular for cographic matroids by Merino in \cite{Merino2001}, lattice-path matroids by Schweig in \cite{Schweig2010}, cotransversal matroids by Oh in \cite{Oh2013} (see also work of Sarmiento \cite{Sar}), paving matroids by Merino, Noble, Ramirez-Ibanez, and Villarroel-Flores  \cite{MNRV2012}, and internally perfect matroids by Dall in \cite{Dal}.  The conjecture has also been shown to hold for small rank and corank, in particular rank $3$ matroids by H\'a, Stokes, and Zanello in \cite{HaStokesZanello2013}, rank $3$ and corank $2$ matroids by DeLoera, Kemper, and Klee in \cite{DelKemKle}, rank $4$ matroids by Klee and Samper in \cite{KleeSamper2015}, and rank $d$ matroids with $h_d \leq 5$  by Constantinescu, Kahle, and Varbaro in \cite{ConKahVar}. Stronger forms of the conjecture have been established for some specific classes of matroids, in particular for rank $3$ matroids in \cite{KleNug} by Klee and Nugent, and for rank $3$ and $4$ matroids in \cite{KleeSamper2015}. %inspired by Las Vergnas' internal order.

%Klee and Samper [Lexicographic shellability, matroids and pure order ideals] \cite{KleeSamper2015} use lexicographic shellability to provide a strengthening of Stanley's conjecture, and prove the conjecture for matroids of rank at most four.

%DeLoera, Kemper, and Klee [$h$-Vectors of Small Matroid Complexes]\cite{DelKemKle} proved
%the conjecture for matroids of rank $3$ and corank $2$ via combinatorial methods.

%Constantinescu, Kahle and Varbaro [Generic and special constructions of pure O-sequences] \cite{ConKahVar} established the conjecture for proper skeleta of matroids and rank $d$ matroids with $h_d leq 5$ using again commutative algebra and level Artinian algebras

%Klee and Nugent [Pure ${\mathcal O}$-sequences arising from 2-dimensional PS ear-decomposable simplicial complexes]\cite{KleNug} proved that $h$-vector of a 2-dimensional PS ear-decomposable simplicial complex is a pure ${\mathcal O}$-sequence. This provides a strengthening of Stanley's conjecture for matroid h-vectors in rank 3. 

In \cite{Kook2012} Kook established Stanley's conjecture for the graphic matroid of a \emph{coned} graph, by definition a graph $\hat G = G \ast \{v\}$ obtained from connecting a vertex $v$ to every vertex of an arbitrary finite graph $G$.  Kook proved the conjecture by explicitly constructing a multicomplex of `partially edge-rooted forests' in $G$.  A spanning tree $T$ of $\hat G$ corresponds to a partially edge-rooted forest of $G$ in such a way that the number of internally passive edges in $T$ is given by the cardinality of edges and edge roots in its corresponding partially edge-rooted forest.  

In \cite{KookLee2018} Kook and Lee studied the $h$-vectors of complete bipartite graphs $K_{m+1, n+1}$ and provided a combinatorial interpretation for their M\"obius coinvariant $\mu^\perp(K_{m+1,n+1})$, which can be seen to coincide with the last nonzero entry of the $h$-vector of the underlying matroid. They showed that the set of such trees correspond to certain edge-rooted forests of the subgraph $K_{m,n}$.  These constructions provide bijective combinatorial proofs for the formulas for $\mu^\perp(K_{m+1})$ and $\mu^\perp(K_{m+1,n+1})$ previously established by Novik, Postnikov, and Sturmfels in \cite{NovPosStu2002}.

In this paper we study $h$-vectors of biconed graphs. By definition a \emph{biconed graph} $\G$ has a pair of vertices $0$ and $\ol{0}$ such that every vertex in $\G$ is adjacent to one of $0$ or $\ol{0}$ (or both).  Loops and some, but not all, parallel edges are admissible (see Definition \ref{biconing} for a precise statement, and in particular the meaning of $A$ and $B$).  The class of biconed graphs includes coned graphs, complete multipartite graphs, and Ferrers graphs.

In the concluding section of \cite{KookLee2018} the authors suggest biconed graphs as a class of graphs for which their `edge-rooted forests' may naturally generalize. In this paper we confirm this, showing that the set of completely passive spanning trees of a biconed graph $\G$ is in correspondence with the collection of maximal `2-weighted forests' of $G^{A,B}_{\text{red}}$, a certain `reduced' subgraph of $\G$. 

Furthermore, we show that by allowing for partially weighted forests, this construction gives rise to a notion of `degree' (in terms of the number of edge roots) in such a way that that the number of internally passive edges in a spanning tree of $\G$ is given by the degree in the corresponding partially 2-weighted forest. Our main results can be summarized as follows.  We refer to later sections for technical definitions.

\begin{theorem}[Corollary \ref{cor:mainbijection}, Lemma \ref{lem:activity}]
Suppose $\G$ is a biconed graph with $h$-vector $(h_0, h_1, \dots, h_d)$. Then $h_i$ is given by the number of partially $2$-weighted forests in $G^{A,B}_{\text{red}}$ of degree $i$.
\end{theorem}

We let $\F(\red)$ denote the set of partially 2-weighted forests in $G = A \cup B$.   The set $\F(\red)$ has a pleasing combinatorial structure, as our next result indicates.

\begin{theorem}[Lemma \ref{lem:multi}, Lemma \ref{lem:pure}]
For any biconed graph $\G$ the set $\F(\red)$ is a pure multicomplex on the set of edges of $G^{A,B}_{\text{red}}$.
\end{theorem}

From these we obtain our main result.

\begin{theorem}[Corollary \ref{cor:Stanley}]
Stanley's conjecture holds for graphic matroids of biconed graphs.
\end{theorem}

In \cite{KleeSamper2015} Klee and Samper introduce a combinatorial strengthening of Stanley's conjecture, based on lexicographic shellability and the notion of a based matroid.  Their conjecture involves the construction of a rule for assigning a pure multicomplex to any based matroids, satisfying a list of five properties.  In Theorem \ref{thm:KleeSamper} we show that our constructions satisfy four of these properties for based matroids arising from biconed graphs.  This in turn connects the multicomplex $\F(\red)$ to the \emph{internal poset} of the matroid of a biconed graph $\G$.

The rest of the paper is organized as follows. In Section \ref{sec:Prelim} we recall some basic notions from matroid theory and the study of pure ${\mathcal O}$-sequences, and establish some notation.  In Section \ref{sec:biconed} we describe our main objects of study and establish bijections between three sets: spanning trees of a biconed graph $\G$, birooted forests in $\red$, and 2-weighted forests in $\red$. In Section \ref{sec:multi}, we prove that the set of partially 2-weighted forests is a pure multicomplex.  Here we also prove that the pure ${\mathcal O}$-sequence arising from this multicomplex is the  $h$-vector of the (graphic matroid of the) underlying biconed graph, thus establishing Stanley's conjecture. In Section \ref{sec:Example} we provide a detailed example to demonstrate our various constructions. In Section \ref{sec:Klee-Samper} we discuss how our results relate to the stronger conditions of the Klee-Samper conjecture for this class of matroids. In Section \ref{sec:Future}, we suggest some further applications of 2-weighted forests and also discuss some open questions.
%Preston: I updated this to explain the sections in the order that they appear in.

%%%%%%%%%%%%%%%%%%%%%%%%%%%%%%%%%%%%%%%%%%%%%%%%%%%%%%%%%%%%%%%%%%%%

\section{Preliminaries}\label{sec:Prelim}

%In this section we recall some basic definitions and set some notation.

\subsection{Matroids}

We first review some basic notions of matroid theory, referring to  \cite{BrylawskiOxley1992} for more details.  For the purposes of this paper, a \emph{matroid} ${\mathcal M} = (E,{\mathcal I})$ on a finite ground set $E$ is a nonempty collection ${\mathcal I}$ of subsets of $E$ satisfying the following properties:
\begin{enumerate}
    \item If $A \in {\mathcal I}$ and $B \subset A$ then $B \in {\mathcal I}$;
    \item If $A, B \in {\mathcal I}$ and $|A| > |B|$ then there exists some $e \in A \setminus B$ such that $B \cup e \in {\mathcal I}$.
\end{enumerate}

Here, we suppress (and will continue to suppress) the brackets of singleton sets. The collection ${\mathcal I}$ is called the set of \emph{independent sets} of the matroid. We let ${\mathcal B} = {\mathcal B}(M)$ denote the set of \emph{bases} of the matroid, by definition the set of maximal independent sets (under inclusion). The number of elements in any (and hence every) basis of ${\mathcal M}$ is called the \emph{rank} of the matroid.  Given a matroid ${\mathcal M} = (E, {\mathcal I})$ the \emph{dual matroid} ${\mathcal M}^*$ has ground set $E$ and bases given by the complements of the bases of ${\mathcal M}$, so that ${\mathcal B}({\mathcal M}^*) = \{E \setminus B:B \in {\mathcal B}({\mathcal M})\}$. 

%I don't think we need the dual.
%ANTON: I suppose we use the dual notion when we talk about `cographic' and `cotransversal' etc.

An important example of a matroid, particularly relevant for us, comes from graph theory.  If $G$ is a finite connected graph with vertex set $V(G)$ and edge set $E(G)$ (possibly with loops and multiple edges) one defines the \emph{graphic matroid} ${\mathcal M}(G)$ with ground set $E = E(G)$ and independent sets given by acyclic collections of edges. The bases are then spanning trees of $G$, and hence the rank of ${\mathcal M}(G)$ is given by $|V(G)| - 1$.

\subsection{Activity and $h$-vectors}
The collection of independent sets of a matroid form a \emph{simplicial complex} called the \emph{independence complex} of ${\mathcal M}$. Associated to a simplicial complex of dimension $d-1$, and therefore to a matroid of rank $d$, is its \emph{f-vector} $f = (f_{-1}, f_0, \dots, f_{d-1})$, where $f_{i-1}$ is the number of simplices of cardinality $i$.
 The $h$-vector of the independence complex of ${\mathcal M}$ (which we will simply refer to as the $h$-vector of ${\mathcal M}$) encodes the same information as $f$ in a form that is more convenient, especially in algebraic contexts. 
 
 We can define the entries of $h = (h_0, \dots, h_d)$ according to the linear relation
\[\sum_{i=0}^d f_{i-1}(t-1)^{d-i} = \sum_{k=0}^d h_kt^{d-k}.\]
\noindent
The $h$-vector of a simplicial complex is related to a presentation of the \emph{Hilbert} function of its Stanley-Reisner (face) ring, and in the case of matroids encodes combinatorial data regarding any \emph{shelling} of its independence complex.

In the case of a matroid ${\mathcal M}$ the $h$-vector is also related to a certain expression for the Tutte polynomial of ${\mathcal M}$, expressed in terms of \emph{activity} of elements in the collection of bases.  For this we fix a linear ordering $<$ on the ground set $E$ of of the matroid ${\mathcal M}$. Now suppose $B$ is a basis for ${\mathcal M}$. For any element $e \in B$ we say that $e$ is \emph{internally passive} in $B$ if it can be replaced by a smaller element to obtain another basis; that is, if $(B \setminus e) \cup e^\prime$ is a basis of $\M$ for some $e^\prime < e$. We say that $e \in B$ is \emph{internally active} if it is not internally passive, that is, if it cannot be replaced by any smaller element from the ground set to get another basis.
An edge $e \notin B$ is said to be \emph{externally active (passive)} if it is (is not) the smallest element in the unique circuit containing $B \cup e$. 
%One can see that the internally active/passive elements with respect to a basis $B$ in ${\mathcal M}$ are the externally active/passive elements with respect to the basis $E \setminus B$ in ${\mathcal M}^*$.

In the case of a graphic matroid ${\mathcal M} = {\mathcal M}(G)$ these constructions can be more explicitly described.  Suppose $G$ is a connected graph with ordered edge set $E$, and let $T$ be a spanning tree of $G$.  Removing an edge $e \in T$ creates a forest with two components.  The \emph{fundamental bond} $B_G(T,e)$ with respect to $e \in T$ consists of all edges in $G$ that have an endpoint in each of the two components.  Then $e$ is internally active if it is the smallest element (with respect to the fixed ordering $<$) in its fundamental bond $B_G(T,e)$. Dually, if $e \notin T$, then the addition of $e$ creates a \emph{fundamental circuit} $C(T,e)$ which is the minimum dependent set containing $e$ and edges from $T$.  Then $e$ is externally active if it is the smallest in this set of edges.  From \cite{Bjorner1992} we have the following.

\begin{lemma} \cite[Section 7.3]{Bjorner1992} \label{h-vector}
Suppose ${\mathcal M}$ is a matroid of rank $d$ with an arbitrary fixed ordering of the ground set, and let $h = (h_0, h_1, \dots, h_d)$ denote its $h$-vector. Then $h_i$ equals the number of bases with $i$ internally passive elements with respect to the ordering of the ground set. The value of $h_i$ is independent of the choice of the ordering. 
\end{lemma}

For a matroid ${\mathcal M}$ its \emph{Tutte polynomial} is given by
\[T_{\mathcal M}(x,y) = \sum \tau_{i,j} x^iy^j,\]
\noindent
where $\tau_{i,j}$ is the number of bases of ${\mathcal M}$ with $i$ internally active elements and $j$ externally active elements. Hence evaluating the Tutte polynomial at $y=1$ gives a polynomial $T_{\mathcal M}(x,1)$ in one variable $x$ where the coefficient of $x^i$ is given by $h_{d-i}$. We refer to \cite{BrylawskiOxley1992} for more details regarding the Tutte polynomial and external activity.

\subsection{Multicomplexes and (pure) ${\mathcal O}$-sequences}

We next review the notion of ${\mathcal O}$-sequences and purity involved in the statement of Stanley's conjecture. Recall that a \emph{multicomplex} $\Delta$ on a ground set $E = \{e_1, e_2, \dots, e_j\} $ is a collection of multisets of elements from $E$ that is closed under taking subsets: if $\sigma \in \Delta$ and $\tau \subset \sigma$, then $\tau \in \Delta$. 

Equivalently, a multicomplex $\Delta$ may be thought of as a set of monomials $e_1^{a_1} e_2^{a_2} \cdots e_j^{a_j}$ satisfying the property that if $p \in \Delta$ and $q$ divides $p$, then $q \in \Delta$.  In this context, a collection of monomials satisfying this condition will be referred to as an \emph{order ideal}. The multisets in $\Delta$ which are maximal under inclusion are the \emph{facets} of $\Delta$. The size of the largest set in $\Delta$ is called the \emph{degree} of $\Delta$. A multicomplex is \emph{pure} if all its facets have the same cardinality.

A sequence of positive integers $(f_{-1}, f_0, \dots, f_{d-1})$ is said to be an \emph{${\mathcal O}$-sequence} if there exists a multicomplex $\Delta$ with the property that $f_i$ is the number of sets in $\Delta$ with cardinality $i+1$, with $d$ the degree of $\Delta$. The sequence is a \emph{pure ${\mathcal O}$-sequence} if $\Delta$ can be chosen to be a \emph{pure} multicomplex.

\section{Biconed graphs and rooted forests} \label{sec:biconed}

We next turn to our main objects of study.  Here we consider graphs that are finite and undirected but possibly with loops and parallel edges. For a graph $G$ with vertices $u,v \in V(G)$, we use $uv$ to denote the edge $\{u,v\}$. The following construction is similar to what is suggested by Kook and Lee in \cite{KookLee2018}.

\begin{definition}\label{biconing}
Suppose $G$ is a graph with vertex set $V(G)$ and suppose $A$ and $B$ are (not necessarily disjoint) subsets of $V(G)$ satisfying $V(G) = A \cup B$. 
The \emph{biconing} of $G$ with respect to $A$ and $B$ is the addition of two additional vertices $0$ and $\ol{0}$ and edges
\begin{itemize}
    \item 
    $0\ol{0}$, 
    \item $0a$ for all $a \in A$, 
    \item
    $\ol{0}b$ for all $b \in B$. 
    
    \end{itemize}
    We use $\G$ to denote the resulting graph. A \emph{biconed graph} is any graph that is obtained from a biconing.
\end{definition}

See Figures \ref{fig:exgraph} and \ref{fig:exgraphcone} for an example of a graph and its biconing. Examples of biconed graphs include coned graphs, complete multipartite graphs, and Ferrers graphs (see Section \ref{sec:Future}).  The path graph $P_5$ on $5$ vertices is an example of a connected graph that is not biconed.

For a biconed graph $\G$ define $\ol A := B \setminus A$, so that $V(G)$ is the disjoint union of $A$ and $\ol A$. The complete bipartite graphs $K_{m,n}$ investigated in \cite{KookLee2018} are examples of biconed graphs with $A \cap B = \varnothing$, or equivalently satisfying $B = \ol A$. We emphasize that our generalization allows for some vertices of $G$ to be connected to both coning vertices.

We will be interested in the activity of elements in the spanning trees of biconed graphs, and for this we define a total order on the edge set inspired by conventions in \cite{KookLee2018}. Let $m = \abs{A}$ and $n = \abs{\ol{A}}$, and without loss of generality assume that the elements of $A$ and $\ol{A}$ are labeled such that $A = [m] = \{1,\dots,m\}$ and $\ol{A} = [\ol{n}] = \{\ol{1},\dots,\ol{n}\}$. We then order the vertices in the following manner:
$$0<\ol 0<1<2<\dots<m<\ol{1}<\dots<\ol{n}.$$

%For an edge connecting vertices $v_1$ and $v_2$ with $v_1 < v_2$, we label the edge $v_1 v_2$.

For the rest of this paper, we assume that the vertices of $\G$ are ordered according to this convention.
For a biconed graph $\G$ we let $T_0 \in \T(\G)$ denote the spanning tree that contains the edge $0\ol{0}$, all edges of form $0u$ for $u \in A$ and edges $\ol{0}v$ for $v \in \ol{A}$.

We now order the edges of $\G$ so that the edges of $T_0$ represent the first $|T_0|$ smallest elements.  In particular we first order the edges of $T_0$ lexicographically according to the order on the vertices described above: $0 \ol{0}, 01, 02, \dots 0m, \ol{0}\ol{1}, \ol{0}\ol{2}, \dots \ol{0}\ol{n}$; and then order the edges of $G \backslash T_0$ (including any parallel edges) arbitrarily.  We use this ordering to define the activity of edges in any spanning tree of $\G$. Note that $T_0$ is the lexicographically smallest spanning tree of $\G$ and has the property that all edges are internally active.

%We now order the edges $E(\G)$ in a way so that the edges of $T_0$ represent the first $|T_0|$ smallest elements.  In particular we order the edges of $T_0$ lexicographically according to the order on the vertices described above: $0 \ol{0}, 01, 02, \dots 0m, \ol{0}\ol{1}, \ol{0}\ol{2}, \dots \ol{0}\ol{n}$, and then continue to order the edges of $\G \backslash T_0$ according to the same lexicographic order. Note that in this case the edge $\ol{0}\ol{1}$ will before the edge $\ol{0}1$ (if it exists).   Parallel edges (which can only exist among edges that do not contain $0$ nor $\ol{0}$) are ordered arbitrarily relative to this order.  For instance if there are three occurrences of the edge $14$ they can be ordered arbitrarily among themselves, but all three must come after $13$ and before $15$. It is this ordering of the edges of $\G$ that will be used to define the activity of edges in any spanning tree. Note that $T_0$ is the lexicographically smallest spanning tree of $\G$ and has the property that all edges are internally active.

In our work we will primarily be interested in subgraphs obtained by removing the edges of $T_0$. We let $\GG$ denote the graph obtained from $\G$ by removing the edges of $T_0$ and also removing the vertex $0$. See Figure \ref{fig:exgraphreduce} for an illustration.

\begin{figure}[h]
\begin{multicols}{3}
\renewcommand{\thefigure}{1A}

\begin{center}
%\textcolor{white}{White}\newline\newline\newline
        \begin{tikzpicture}[scale=.6]
        \node [draw,fill,circle,inner sep = 0pt, minimum size = .15cm,label=below:{\small1}] (M1) at (-.5,0) {};
        \node [draw,fill,circle,inner sep = 0pt, minimum size = .15cm,label=below:{\small2}] (M2) at (.5,0) {};
        \node [draw,fill,circle,inner sep = 0pt, minimum size = .15cm,label=below:{\small3}] (M3) at (1.5,0) {};
        \node [draw,fill,circle,inner sep = 0pt, minimum size = .15cm,label=below:{\small4}] (M4) at (2.5,0) {};
        \node [draw,fill,circle,inner sep = 0pt, minimum size = .15cm,label=above:{\small$\ol{1}$}] (N1) at (0,2) {};
        \node [draw,fill,circle,inner sep = 0pt, minimum size = .15cm,label=above:{\small$\ol{2}$}] (N2) at (1,2) {};
        \node [draw,fill,circle,inner sep = 0pt, minimum size = .15cm,label=above:{\small$\ol{3}$}] (N3) at (2,2) {};
        \node [draw,fill,circle,inner sep = 0pt, minimum size = .15cm,label=above:{\small$\ol{4}$}] (N4) at (3,2) {};
        \node [draw,fill,circle,inner sep = 0pt, minimum size = .15cm,label=below:{\small5}] (M5) at (3.5,0) {};
        \node [draw,fill,circle,inner sep = 0pt, minimum size = .15cm,label=below:{\small6}] (M6) at (4.5,0) {};
        \node [draw,fill,circle,inner sep = 0pt, minimum size = .15cm,label=above:{\small$\ol{5}$}] (N5) at (4,2) {};
        
        %invisible coning points
        \node [inner sep = 0pt, minimum size = .15cm,label=below:{\small\textcolor{white}{0}}] (C0) at (-.8,-1.35) {};
        \node [draw=white, fill=white,circle=white,inner sep = 0pt, minimum size = .25cm,label=above:{\small\textcolor{white}{0}}] (C1) at (-.8,3.35) {};

    \draw (M1) -- (N1);
    \draw (M1) -- (M2);
    \draw (M2) -- (N5);
    \draw (M3) -- (N1);   
    \draw (M3) -- (N5);
    \draw (M4) -- (N4);
    \draw (M5) -- (M6);
    \draw (M6) -- (N5);   
    \draw (N1) -- (N2);
    \draw (N3) -- (N4);
    \draw (M1) -- (N3);
    \draw (M5) -- (N4);
    \draw (M2) -- (N1);
    \end{tikzpicture}
\captionsetup{width=0.85\linewidth}
 \captionof{figure}{The graph $G$ with $V(G) = A \cup B$, where $A = \{1, \dots, 6\}$ and $B = \{1,2,\ol{1}, \dots, \ol{5}\}$.}
  \label{fig:exgraph}
\end{center}

\vfill\null
\columnbreak

\renewcommand{\thefigure}{1B}

\begin{center}

             \begin{tikzpicture}[scale=.6]
        \node [draw,fill,circle,inner sep = 0pt, minimum size = .15cm,label=below:{\small1}] (M1) at (-.5,0) {};
        \node [draw,fill,circle,inner sep = 0pt, minimum size = .15cm,label=below:{\small2}] (M2) at (.5,0) {};
        \node [draw,fill,circle,inner sep = 0pt, minimum size = .15cm,label=below:{\small3}] (M3) at (1.5,0) {};
        \node [draw,fill,circle,inner sep = 0pt, minimum size = .15cm,label=below:{\small4}] (M4) at (2.5,0) {};
        \node [draw,fill,circle,inner sep = 0pt, minimum size = .15cm,label=above:{\small$\ol{1}$}] (N1) at (0,2) {};
        \node [draw,fill,circle,inner sep = 0pt, minimum size = .15cm,label=above:{\small$\ol{2}$}] (N2) at (1,2) {};
        \node [draw,fill,circle,inner sep = 0pt, minimum size = .15cm,label=above:{\small$\ol{3}$}] (N3) at (2,2) {};
        \node [draw,fill,circle,inner sep = 0pt, minimum size = .15cm,label=above:{\small$\ol{4}$}] (N4) at (3,2) {};
        \node [draw,fill,circle,inner sep = 0pt, minimum size = .15cm,label=below:{\small5}] (M5) at (3.5,0) {};
        \node [draw,fill,circle,inner sep = 0pt, minimum size = .15cm,label=below:{\small6}] (M6) at (4.5,0) {};
        \node [draw,fill,circle,inner sep = 0pt, minimum size = .15cm,label=above:{\small$\ol{5}$}] (N5) at (4,2) {};
        
        %cone points
        \node [draw,fill,circle,inner sep = 0pt, minimum size = .15cm,label=below:{\small0}] (C0) at (-1,-1.35) {};
        \node [draw,fill,circle,inner sep = 0pt, minimum size = .15cm,label=above:{\small$\ol{0}$}] (C1) at (-1,3.35) {};
        
    \draw(C0) -- (C1);   
    \draw(C0) -- (M1);
    \draw (C0) -- (M2);
    \draw  (C0) -- (M3);
    \draw  (C0) -- (M4);
    \draw  (C0) -- (M5);
    \draw  (C0) -- (M6);
     \draw  (C1) -- (N1);
    \draw  (C1) -- (N2);
    \draw  (C1) -- (N3);
    \draw  (C1) -- (N4);
    \draw  (C1) -- (N5);
     \draw (M1) -- (N1);
    \draw  (C1) -- (M1);
    \draw  (C1) -- (M2);
    \draw (M1) -- (M2);
    \draw (M2) -- (N5);
    \draw (M3) -- (N1);   
    \draw (M3) -- (N5);
    \draw (M4) -- (N4);
    \draw (M5) -- (M6);
    \draw (M6) -- (N5);   
    \draw (N1) -- (N2);
    \draw (N3) -- (N4);
    \draw (M1) -- (N3);
    \draw (M5) -- (N4);
    \draw (M2) -- (N1);

    \end{tikzpicture}
    \captionsetup{width=0.85\linewidth}
    \captionof{figure}{The resulting biconed graph $\G$.}
    \label{fig:exgraphcone}
    \end{center}
%Preston: I've updated figure 1b so that the edges of T_0 are dotted. They were previously bolded.
%%%%%%%%%%%%%%%%%%%%%%%%%%%%%%%%%%%%%%%
%I'm going to try to make a new one

\vfill\null
\columnbreak

\renewcommand{\thefigure}{1C}

\begin{center}

             \begin{tikzpicture}[scale=.6]
        \node [draw,fill,circle,inner sep = 0pt, minimum size = .15cm,label=below:{\small1}] (M1) at (-.5,0) {};
        \node [draw,fill,circle,inner sep = 0pt, minimum size = .15cm,label=below:{\small2}] (M2) at (.5,0) {};
        \node [draw,fill,circle,inner sep = 0pt, minimum size = .15cm,label=below:{\small3}] (M3) at (1.5,0) {};
        \node [draw,fill,circle,inner sep = 0pt, minimum size = .15cm,label=below:{\small4}] (M4) at (2.5,0) {};
        \node [draw,fill,circle,inner sep = 0pt, minimum size = .15cm,label=above:{\small$\ol{1}$}] (N1) at (0,2) {};
        \node [draw,fill,circle,inner sep = 0pt, minimum size = .15cm,label=above:{\small$\ol{2}$}] (N2) at (1,2) {};
        \node [draw,fill,circle,inner sep = 0pt, minimum size = .15cm,label=above:{\small$\ol{3}$}] (N3) at (2,2) {};
        \node [draw,fill,circle,inner sep = 0pt, minimum size = .15cm,label=above:{\small$\ol{4}$}] (N4) at (3,2) {};
        \node [draw,fill,circle,inner sep = 0pt, minimum size = .15cm,label=below:{\small5}] (M5) at (3.5,0) {};
        \node [draw,fill,circle,inner sep = 0pt, minimum size = .15cm,label=below:{\small6}] (M6) at (4.5,0) {};
        \node [draw,fill,circle,inner sep = 0pt, minimum size = .15cm,label=above:{\small$\ol{5}$}] (N5) at (4,2) {};
        
        %cone points

        \node [draw,fill,circle,inner sep = 0pt, minimum size = .15cm,label=above:{\small$\ol{0}$}] (C1) at (-1,3.35) {};
        %this one is invisible for spacing
        \node [inner sep = 0pt, minimum size = .15cm,label=below:{\small\textcolor{white}{0}}] (C0) at (-.8,-1.35) {};
        
    \draw (C1) -- (M2);
    \draw (C1) -- (M1);
    \draw (M1) -- (M2);
    \draw (M1) -- (N1);
    \draw (M2) -- (N5);
    \draw (M3) -- (N1);   
    \draw (M3) -- (N5);
    \draw (M4) -- (N4);
    \draw (M5) -- (M6);
    \draw (M6) -- (N5);   
    \draw (N1) -- (N2);
    \draw (N3) -- (N4);
    \draw (M1) -- (N3);
    \draw (M5) -- (N4);
    \draw (M2) -- (N1);

    \end{tikzpicture}
    \captionsetup{width=0.85\linewidth}
    \captionof{figure}{The subgraph $G^{A,B}_{\text{red}}$, obtained by removing the edges of $T_0$ and the vertex $0$.}
    \label{fig:exgraphreduce}
    \end{center}
    
\end{multicols}
\end{figure}

\subsection{Birooted forests}
For a biconed graph $\G$ we let $\T(\G)$ denote its set of spanning trees.  We wish to encode the elements in $\T(\G)$ in terms of more convenient combinatorial structures. For this we'll need the following notions.

\begin{definition}
Suppose $T \in \T(\G)$ is a spanning tree of a biconed graph $\G$. A vertex $v\in T$ is a \emph{connecting vertex} if $v\in A$ and is adjacent to $0$ or $v \in \ol A$ and is adjacent to $\ol 0$. A \emph{connecting edge} in $T$ is an edge of the form $0u$ or $\ol{0}v$ for $u \in A$ or $v \in \ol{A}$.
\end{definition}

See Figure \ref{fig:exampletree} for an example of these concepts.  We will see that connecting edges and vertices encode activity of edges in a spanning tree of a biconed graph. In what follows we think of a tree as a collection of edges, so that if $S$ and $T$ are trees we use $T \backslash S$ to denote the difference of edge sets $E(T) \backslash E(S)$. We then have the following observation, which motivates our constructions.

%\begin{lemma}\label{lem:active}
%Suppose $T \in \T(\G)$ is a spanning tree of $\G$.  Then the internally active edges of $T$ are given by the edge $0\ol{0}$ (if it exists in $T$) and all connecting edges involving a connecting vertex $v$, where $v$ is the largest vertex in its component within $T \setminus T_0$.
%\end{lemma}
%original version had a couple typos

\begin{lemma}\label{lem:active}\label{lem:passive}
Suppose $T \in \T(\G)$ is a spanning tree of a biconed graph $\G$.  Then the internally passive edges of $T$ are given by all edges in $T \setminus T_0$, as well as all connecting edges involving a connecting vertex $v$ where
\begin{itemize}
    \item $v$ is not the smallest vertex in its component $C_v$ within $T \setminus T_0$, or 
    \item there exists a connecting vertex $w \in C_v$ with $w \neq v$.
    \end{itemize}
\end{lemma}

\begin{example}
Before proving Lemma \ref{lem:passive}, we illustrate these concepts with the graph and spanning tree depicted in Figures \ref{fig:examplegraph} and \ref{fig:exampletree}.  Here the edges of $T \setminus T_0$ consist of edges $\ol{0}1$, $2\ol{1}$, $4\ol{4}$, $5\ol{4}$, $6\ol{5}$, $\ol{1}\ol{2}$, and $\ol{3}\ol{4}$, which are all passive since they can be replaced with smaller connecting edges to obtain another spanning tree.  The connecting edges are $02$, $03$, $06$, $\ol{0}\ol{2}$, and $\ol{0}\ol{4}$.  Among these, $\ol{0}\ol{4}$ is passive according to the first criterion ($\ol{4}$ is not the smallest in its component after deleting $T_0$).  We see that $02$  and $\ol{0}\ol{2}$ are passive according to the second criterion, since there are two connecting vertices in the component. Finally, $03$ and $06$ are not passive as the respective connecting vertices are the smallest in their components.
\end{example}

\begin{figure}[h]
\begin{multicols}{2}
\renewcommand{\thefigure}{2A}
 \begin{center}
 
             \begin{tikzpicture}[scale=.7]
        \node [draw,fill,circle,inner sep = 0pt, minimum size = .15cm,label=below:{\small1}] (M1) at (-.5,0) {};
        \node [draw,fill,circle,inner sep = 0pt, minimum size = .15cm,label=below:{\small2}] (M2) at (.5,0) {};
        \node [draw,fill,circle,inner sep = 0pt, minimum size = .15cm,label=below:{\small3}] (M3) at (1.5,0) {};
        \node [draw,fill,circle,inner sep = 0pt, minimum size = .15cm,label=below:{\small4}] (M4) at (2.5,0) {};
        \node [draw,fill,circle,inner sep = 0pt, minimum size = .15cm,label=above:{\small$\ol{1}$}] (N1) at (0,2) {};
        \node [draw,fill,circle,inner sep = 0pt, minimum size = .15cm,label=above:{\small$\ol{2}$}] (N2) at (1,2) {};
        \node [draw,fill,circle,inner sep = 0pt, minimum size = .15cm,label=above:{\small$\ol{3}$}] (N3) at (2,2) {};
        \node [draw,fill,circle,inner sep = 0pt, minimum size = .15cm,label=above:{\small$\ol{4}$}] (N4) at (3,2) {};
        \node [draw,fill,circle,inner sep = 0pt, minimum size = .15cm,label=below:{\small5}] (M5) at (3.5,0) {};
        \node [draw,fill,circle,inner sep = 0pt, minimum size = .15cm,label=below:{\small6}] (M6) at (4.5,0) {};
        \node [draw,fill,circle,inner sep = 0pt, minimum size = .15cm,label=above:{\small$\ol{5}$}] (N5) at (4,2) {};
        
        %cone points
        \node [draw,fill,circle,inner sep = 0pt, minimum size = .15cm,label=below:{\small0}] (C0) at (-1,-1.35) {};
        \node [draw,fill,circle,inner sep = 0pt, minimum size = .15cm,label=above:{\small$\ol{0}$}] (C1) at (-1,3.35) {};
        
    \draw[line width=1.35pt] (C0) -- (C1);   
    \draw[line width=1.35pt]  (C0) -- (M1);
    \draw[line width=1.35pt]  (C0) -- (M2);
    \draw[line width=1.35pt]  (C0) -- (M3);
    \draw[line width=1.35pt]  (C0) -- (M4);
    \draw[line width=1.35pt]  (C0) -- (M5);
    \draw[line width=1.35pt]  (C0) -- (M6);
     \draw[line width=1.35pt]  (C1) -- (N1);
    \draw[line width=1.35pt]  (C1) -- (N2);
    \draw[line width=1.35pt]  (C1) -- (N3);
    \draw[line width=1.35pt]  (C1) -- (N4);
    \draw[line width=1.35pt]  (C1) -- (N5);
     \draw (M1) -- (N1);
    \draw  (C1) -- (M1);
    \draw  (C1) -- (M2);
    \draw (M1) -- (M2);
    \draw (M2) -- (N5);
    \draw (M3) -- (N1);   
    \draw (M3) -- (N5);
    \draw (M4) -- (N4);
    \draw (M5) -- (M6);
    \draw (M6) -- (N5);   
    \draw (N1) -- (N2);
    \draw (N3) -- (N4);
    \draw (M1) -- (N3);
    \draw (M5) -- (N4);
    \draw (M2) -- (N1);

    \end{tikzpicture}
       \captionsetup{width=1.0\linewidth}
  \captionof{figure}{Our running example of a biconed graph $\G$, with edges of $T_0$ in bold.}
  \label{fig:examplegraph}
    
 \end{center}

%%%%%%%%%%%%%%%%%%%%%%%%%%%%%%%%%%%%%%%%%%%%%%%%%%%%%%%%%%%%%%%%%%%
\vfill\null
\columnbreak
%%%%%%%%%%%%%%%%%%%%%%%%%%%%%%%%%%%%%%%%%%%%%%%%%%%%%%%%%%%%%%%%%%%
\renewcommand{\thefigure}{2B}
 \begin{center}
     \begin{tikzpicture}[scale=.7]
    %\centering
        \node [draw,fill,circle,inner sep = 0pt, minimum size = .15cm,label=below:{\small1}] (M1) at (-.5,0) {};
        \node [draw,fill,circle,inner sep = 0pt, minimum size = .21cm,label=below:{\small2}] (M2) at (.5,0) {};
        \node [draw,fill,circle,inner sep = 0pt, minimum size = .21cm,label=below:{\small3}] (M3) at (1.5,0) {};
        \node [draw,fill,circle,inner sep = 0pt, minimum size = .15cm,label=below:{\small4}] (M4) at (2.5,0) {};
        \node [draw,fill,circle,inner sep = 0pt, minimum size = .15cm,label=above:{\small$\ol{1}$}] (N1) at (0,2) {};
        \node [draw,fill,circle,inner sep = 0pt, minimum size = .21cm,label=above:{\small$\ol{2}$}] (N2) at (1,2) {};
        \node [draw,fill,circle,inner sep = 0pt, minimum size = .15cm,label=above:{\small$\ol{3}$}] (N3) at (2,2) {};
        \node [draw,fill,circle,inner sep = 0pt, minimum size = .21cm,label=above:{\small$\ol{4}$}] (N4) at (3,2) {};
        \node [draw,fill,circle,inner sep = 0pt, minimum size = .15cm,label=below:{\small5}] (M5) at (3.5,0) {};
        \node [draw,fill,circle,inner sep = 0pt, minimum size = .21cm,label=below:{\small6}] (M6) at (4.5,0) {};
        \node [draw,fill,circle,inner sep = 0pt, minimum size = .15cm,label=above:{\small$\ol{5}$}] (N5) at (4,2) {};
        
        %cone points
        \node [draw,fill,circle,inner sep = 0pt, minimum size = .15cm,label=below:{\small0}] (C0) at (-.8,-1.35) {};
        \node [draw,fill,circle,inner sep = 0pt, minimum size = .15cm,label=above:{\small$\ol{0}$}] (C1) at (-.8,3.35) {};

    \draw[line width=1.35pt] (C0) -- (M2);
    \draw[line width=1.35pt] (C0) -- (M3);
    \draw[line width=1.35pt] (C0) -- (M6);
    \draw[line width=1.35pt] (C1) -- (N2);
    \draw[line width=1.35pt] (C1) -- (N4);
    \draw (N1) -- (M2);
    \draw (M4) -- (N4);
    \draw (M6) -- (N5);
    \draw (M5) -- (N4);   
    \draw (N1) -- (N2);
    \draw (N3) -- (N4);
    \draw (C1) -- (M1);
    \end{tikzpicture}
    \captionsetup{width=1.0\linewidth}
  \captionof{figure}{
     A spanning tree $T$ of $\G$, with its connecting edges and vertices in bold.}

  \label{fig:exampletree}
 \end{center}

\end{multicols}
\label{fig:example}
\end{figure}

\begin{proof}[Proof of Lemma \ref{lem:passive}]
First suppose that $e$ is an edge in $T \setminus T_0$. By the matroid exchange property there exists an edge $f$ in $T_0$ such that $(T \backslash e) \cup f$ is a spanning tree of $\G$. Note that $f < e$ since all elements of $T_0$ are smaller than $e$.  We conclude that $e$ is passive.

%Let $C_v$ denote the component in $T \backslash T_0$ that contains $xv$. If $C_v$ contains one connecting vertex then removing $xv$ creates a component with no connecting vertex. We can then replace $xv$ with an edge from $0$, or from $\ol{0}$ if this is not possible, to that component, and this edge will be smaller than $xv$.  If $C_v$ contains two connecting vertices then necessarily $0\ol{0} \notin T$, and $C_v$ is the only component of $T \setminus T_0$ that contains two connecting vertices (since otherwise we would have a cycle). Removing $xv$ either creates two components with one connecting vertex each, in which case we can replace $xv$ by $0 \ol 0$, or creates a component with no connecting vertices, in which case we can replace $xv$ with an edge from $0$ or $\ol{0}$ to that component, and this edge will be smaller than $xv$.  We conclude that $xv$ is internally passive.

Next suppose that $xv$ is a connecting edge with $x = 0$ or $x = \ol{0}$, that $v$ is the only connecting vertex in its component, and that $u$ is a vertex in $C_v$ which is smaller than $u$.  If $u \in A$, we may replace $xv$ with $0u$; if $u \in \ol{A}$, we may replace $xv$ with $\ol{0}u$; lastly, if $u = \ol{0}$, we may replace $xv$ with $0 \ol{0}$ (in this last case, $v$ must be in $A$).  In all cases, $xv$ is internally passive.

Now suppose $xv$ is a connecting edge, again with $x = 0$ or $x = \ol{0}$, such that the component $C_v \subset T \setminus T_0$ contains a distinct connecting vertex $w \neq v$. We must have that $0\ol{0} \notin T$ and also that $C_v$ is the only component in $T \setminus T_0$ that contains two connecting vertices (since otherwise $T$ would contain a cycle).  We can then replace $xv$ with the smaller $0\ol{0}$ to obtain a spanning tree.  

To see that these are the only internally passive elements, first note that  $0\ol{0}$ is always internally active if present as it's the smallest edge in the underlying graph.  Suppose that $xv$ is a connecting edge in $T$ (with $x = 0$ or $x = \ol{0}$) where $v$ the smallest vertex in $C_v$ and also the only connecting vertex in $C_v$. We must then have $\ol{0} \notin C_v$. If we remove $xv$, to replace it with a smaller edge, we would need to replace it with an edge from $T_0$ that connects a vertex in $C_v$ to $0$ or $\ol 0$. But since $v$ is the smallest vertex in $C_v$ it is not possible to do this with an edge that is smaller that $xv$. We conclude that $xv$ is internally active. The result follows.
\end{proof}

For a biconed graph $\G$, recall that $\GG$ is the graph obtained from $\G$ by removing the edges of $T_0$ and also removing the vertex $0$. Notice that for any spanning tree $T$ of $\G$, the graph $T \setminus T_0$ in provides a subforest of $\GG$. This process of course loses information, but we will attach just enough auxiliary data to such a  subforest to encode all the activity information of the original spanning tree $T$.  See Figure \ref{fig:birooted} for a preview of the construction.

To make this precise we extend some notions from Kook and Lee \cite{KookLee2018}, where the case of complete bipartite graphs was studied.  Suppose $F$ is a (not necessarily maximal) spanning forest of $\GG$. We emphasize that $F$ must contain all the vertices of $\GG$, but need not be a maximal acyclic subgraph of $\GG$, and should be thought of as an acyclic collection of edges of $\GG$ in which all vertices are included. A \emph{rooted component} of $F$ is a component in the forest which has exactly one vertex marked as a \emph{root vertex}. A \emph{birooted component} of $F$ is a component which has $2$ vertex roots, with one root in $A$ and the other in $\ol A \cup \ol{0}$. 

\begin{definition}
Given a biconed graph $G^{A,B}$, a \emph{birooted forest} (of $\GG$) is a spanning forest $F$ of $\GG$ such that $\ol{0}$ is a root vertex, at most one component is birooted, and every other component of $F$ is rooted.
\end{definition}

See Figure \ref{fig:birooted} for an example of birooted forest.  Recall that $\T(\G)$ denotes the set of all spanning trees of a biconed graph $\G$.  We let $\R(\GG)$ denote the set of birooted forests of $G^{A,B}_{\text{red}}$. Our next result relates these two sets.

\begin{lemma}
\label{birooted}
For any biconed graph $\G$ there exists a bijection $\phi_1:\T(\G) \rightarrow \R(\GG)$.
\end{lemma}

\begin{proof}
Suppose $T \in \T(\G)$ is a spanning tree.  The edges that get deleted from $T$ as we move to the forest $T \setminus T_0$ are edges $0\ol{0}$, edges of the form $0a$ for $a \in A$ and $\ol{0}b$ for $b \in \ol{A}$ (here, $a$ and $b$ are connecting vertices of $T$). To define $\phi_1(T)$ as a spanning forest of $\GG$, we simply root all the connecting vertices. Every component in this forest must have a rooted vertex since $T$ is connected and spanning. Additionally, there is at most one birooted component in $\phi_1(T)$ since $T$ is acyclic. After all of this is done, root the vertex $\ol{0}$ to obtain $\phi_1(T)$.

The inverse can also easily be described.  Given a birooted forest $F$, simply convert the rooted vertices to connecting vertices and add the connecting edges accordingly (connecting vertices from $A$ go to $0$ and those from $\ol{A}$ to $\ol{0}$). Finally, in the case that there is no birooted component we add back the edge $0 \ol{0}$.  The resulting graph $T$ is a spanning tree of $\G$ with the property that $\phi_1(T) = F$.
\end{proof}

We refer to Figures \ref{fig:spantree} and \ref{fig:birooted} for an illustration of the map $\phi_1$. At this point the activity information we need from a spanning tree $T$ of $\G$ is stored in a spanning forest of $\GG$ with the help of auxiliary information stored in the rooted vertices.  To construct a multicomplex, we move this information to the edges.  We describe this process in the next subsection.

\subsection{$2$-weighted-forests} 
Here we describe our method of encoding activity of a spanning tree of a biconed graph in terms of weighted edges in our auxiliary construction. Let $\G$ be a biconed graph and suppose $F$ is any set of edges of $\GG$. Let $\omega_F:F \rightarrow {\mathbb N}_{\geq 1}$ denote a positive integer weighting on these set of edges. We will use $\omega$ if the context is clear and  refer to the pair $(F,\omega)$ as a \emph{weighted collection of edges}. An edge $e \in F$ is said to have \emph{weight n} if $\omega_F(e) = n$.

% If $F \subset \GG$ is a set of edges with weighting $\omega_F$, a \emph{weighted component} of $F$ 
For a weighted collection of edges $(F, \omega_F)$, a \emph{weighted component}, or just \emph{component} if clear from context, is a (maximal) connected component $S \subset F$ along with a weighting $\omega_S$ which is equal to $\omega_F$ restricted to $S$.  The \emph{excess weight} of a weighted component $S$ is given by
\[\delta_{\ol{0} \in S} + \sum_{e \in S} (\omega(e) - 1),\] 
where $\delta_{\ol{0} \in S}$ equals $1$ if $\ol{0} \in S$ and equals $0$ otherwise.  We assign weight to $\ol{0}$ in this way because it acts in some ways like an additional weighted edge.  A component $S$ is said to be \emph{$k$-excess weighted} if its excess weight is $k$. Finally, an edge in $\GG$ is \emph{crossing} if it connects a vertex in $A$ to a vertex in $\ol{A} \cup \ol{0}$.

\begin{definition}
\label{def:2erf}
For a biconed graph $\G$, a \emph{2-weighted forest} is a weighted set of edges  $(F, \omega)$ with $F \subset E(\GG)$ satisfying:
\begin{itemize}
    \item [$(C0)$] $F$ induces a forest in $\GG$.
    \item [$(C1)$] $(F, \omega)$ has at most one component with excess weight $2$.
    
    \item [$(C2)$] Every other component of $(F, \omega)$ has excess weight $0$ or $1$.
    
    \item [$(C3)$] In a component of $(F, \omega)$ of excess weight $2$ that does not contain $\ol{0}$, the number of crossing edges in the (unique) shortest path containing the weighted edges is odd. In a component of $(F, \omega)$ of excess weight $2$ that does contain $\ol 0$, the number of crossing edges in the shortest path containing both the weighted edge and $\ol{0}$ is odd.
\end{itemize}
\end{definition}

We refer to Figure \ref{fig:2weighted} for an example of a $2$-weighted forest. It will be useful to talk about the subgraph of $\GG$ induced by the edges of a 2-weighted forest. For the rest of the paper we will use the convention that this associated graph has the edges of $F$ and \textit{all} vertices of of $\GG$, whether or not they are an endpoint of an edge in $F$, hence giving a spanning forest.

\begin{remark}
The notion of a $2$-weighted forest is a generalization of the \emph{edge-rooted forests} from \cite{KookLee2018}, where now multiple edge roots are allowed.  In an earlier version of our paper we referred to these objects as \emph{2-edge-rooted forests} but we found the weighting terminology to be less cumbersome.
The main new idea here is the parity condition $(C3)$,  which ensures that in components of excess weight $2$, the edges carrying this excess weight are, in some sense, on opposite sides of the graph. We will see that this condition ensures that the associated multicomplex is pure.
\end{remark}

For a biconed graph $\G$ we let $\F(\GG)$ denote the set of all $2$-weighted forests of $\GG$.  In Lemma \ref{birooted} we saw that spanning trees can be encoded as birooted forests.  Next we observe that birooted forests are encoded in $2$-weighted forests.  Later we will see how the activity of a given spanning tree can be naturally read off in terms of the corresponding $2$-weighted forest.

\begin{lemma}
\label{2-edge-rooted}
For any biconed graph $\G$ there is a bijection $\phi_2: \R(\GG) \rightarrow \F(\GG)$.
\end{lemma}

\begin{proof}
Let $R$ be a birooted forest of $\G$, with underlying edge set $F$. To define $\phi_2(R)$, we show how to convert the information of rooted vertices of $R$ into a weighting $\omega:F \rightarrow {\mathbb N}_{\geq 1}$ on the edges. For each rooted (but not birooted) component of $R$, let $v_s$ denote the smallest vertex in that component.  If $v_s$ is the rooted vertex then all edges in that component get weight one.  If the rooted vertex is not $v_s$, then assign $\omega(e) = 2$ where $e$ is the first edge of the (unique) path from the rooted vertex to $v_s$.
Next, if there is a birooted component of $R$ containing $\ol{0}$ (which is always rooted), then we assign $\omega(f) = 2$ where $f$ is the first edge in the path from the other root vertex to $\ol{0}$. For any other birooted component, we assign a weight of $2$ to the first and last edges of the path from one rooted vertex to the other. If this path consists of exactly one edge, we assign a weight of $3$ to that edge. Finally, any edge not yet assigned weight is given weight $1$.

We check that the resulting weighted collection of edges is a $2$-weighted forest. Condition $(C0)$ is automatically met as $F$ consists of the same underlying edge set as $R$. Condition $(C1)$ and $(C2)$ are met since there is at most one birooted component in $R$. Condition $(C3)$ is met since in a birooted component, one vertex root lies in $A$ and the other lies in $\ol A \cup \ol 0$, so the path between the root vertices must cross between $A$ and $\ol A \cup \ol 0$ an odd number of times in total. 

To see that the above process is reversible, notice that just by checking the excess weight of a component it is easy to see if the component was rooted or birooted in the birooted forest. If the component has no excess weight, we give $v_s$ the root. If there is a single weighted edge $e$ we root the vertex of $e$ that is furthest from $v_s$ in that component.  If the component has excess weight $2$ we recover the desired vertex roots by considering the unique shortest path that connects the two weighted edges (or $\ol{0}$ and the single weighted edge). Finally if $v$ is a vertex in $\G$ not used in any edge in $F$ we root that isolated vertex.
\end{proof}

\begin{figure}[h]
\begin{multicols}{3}
\renewcommand{\thefigure}{3A}
 \begin{center}
 
        \begin{tikzpicture}[scale=.6]\label{Fig 1A}
    %\centering
        \node [draw,fill,circle,inner sep = 0pt, minimum size = .15cm,label=below:{\small1}] (M1) at (-.5,0) {};
        \node [draw,fill,circle,inner sep = 0pt, minimum size = .225cm,label=below:{\small2}] (M2) at (.5,0) {};
        \node [draw,fill,circle,inner sep = 0pt, minimum size = .225cm,label=below:{\small3}] (M3) at (1.5,0) {};
        \node [draw,fill,circle,inner sep = 0pt, minimum size = .15cm,label=below:{\small4}] (M4) at (2.5,0) {};
        \node [draw,fill,circle,inner sep = 0pt, minimum size = .15cm,label=above:{\small$\ol{1}$}] (N1) at (0,2) {};
        \node [draw,fill,circle,inner sep = 0pt, minimum size = .225cm,label=above:{\small$\ol{2}$}] (N2) at (1,2) {};
        \node [draw,fill,circle,inner sep = 0pt, minimum size = .15cm,label=above:{\small$\ol{3}$}] (N3) at (2,2) {};
        \node [draw,fill,circle,inner sep = 0pt, minimum size = .225cm,label=above:{\small$\ol{4}$}] (N4) at (3,2) {};
        \node [draw,fill,circle,inner sep = 0pt, minimum size = .15cm,label=below:{\small5}] (M5) at (3.5,0) {};
        \node [draw,fill,circle,inner sep = 0pt, minimum size = .225cm,label=below:{\small6}] (M6) at (4.5,0) {};
        \node [draw,fill,circle,inner sep = 0pt, minimum size = .15cm,label=above:{\small$\ol{5}$}] (N5) at (4,2) {};
        
        %cone points
        \node [draw,fill,circle,inner sep = 0pt, minimum size = .15cm,label=below:{\small0}] (C0) at (-.8,-1.35) {};
        \node [draw,fill,circle,inner sep = 0pt, minimum size = .15cm,label=above:{\small$\ol{0}$}] (C1) at (-.8,3.35) {};

    \draw[line width=1.5pt] (C0) -- (M2);
    \draw[line width=1.5pt] (C0) -- (M3);
    \draw[line width=1.5pt] (C0) -- (M6);
    \draw[line width=1.5pt] (C1) -- (N2);
    \draw[line width=1.5pt] (C1) -- (N4);
    \draw (N1) -- (M2);
    \draw (M4) -- (N4);
    \draw (M6) -- (N5);
    \draw (M5) -- (N4);   
    \draw (N1) -- (N2);
    \draw (N3) -- (N4);
    \draw (C1) -- (M1);
    \end{tikzpicture}
    \captionsetup{width=1.0\linewidth}
  \captionof{figure}{
     A spanning tree $T$ of $\G$, with its connecting edges and vertices in bold.}
  \label{fig:spantree}
 \end{center}

%%%%%%%%%%%%%%%%%%%%%%%%%%%%%%%%%%%%%%%%%%%%%%%%%%%%%%%%%%%%%%%%%%%
\vfill\null
\columnbreak
%%%%%%%%%%%%%%%%%%%%%%%%%%%%%%%%%%%%%%%%%%%%%%%%%%%%%%%%%%%%%%%%%%%
\renewcommand{\thefigure}{3B}
 \begin{center}
 
        \begin{tikzpicture}[scale=.6]
    %\centering
        \node [draw,fill,circle,inner sep = 0pt, minimum size = .15cm,label=below:{\small1}] (M1) at (-.5,0) {};
        \node [draw,fill,circle,inner sep = 0pt, minimum size = .25cm,label=below:{\small2}] (M2) at (.5,0) {};
        \node [draw,fill,circle,inner sep = 0pt, minimum size = .25cm,label=below:{\small3}] (M3) at (1.5,0) {};
        \node [draw,fill,circle,inner sep = 0pt, minimum size = .15cm,label=below:{\small4}] (M4) at (2.5,0) {};
        \node [draw,fill,circle,inner sep = 0pt, minimum size = .15cm,label=above:{\small$\ol{1}$}] (N1) at (0,2) {};
        \node [draw,fill,circle,inner sep = 0pt, minimum size = .25cm,label=above:{\small$\ol{2}$}] (N2) at (1,2) {};
        \node [draw,fill,circle,inner sep = 0pt, minimum size = .15cm,label=above:{\small$\ol{3}$}] (N3) at (2,2) {};
        \node [draw,fill,circle,inner sep = 0pt, minimum size = .25cm,label=above:{\small$\ol{4}$}] (N4) at (3,2) {};
        \node [draw,fill,circle,inner sep = 0pt, minimum size = .15cm,label=below:{\small5}] (M5) at (3.5,0) {};
        \node [draw,fill,circle,inner sep = 0pt, minimum size = .25cm,label=below:{\small6}] (M6) at (4.5,0) {};
        \node [draw,fill,circle,inner sep = 0pt, minimum size = .15cm,label=above:{\small$\ol{5}$}] (N5) at (4,2) {};

        %cone points
        \node [inner sep = 0pt, minimum size = .15cm,label=below:{\small\textcolor{white}{0}}] (C0) at (-.8,-1.35) {};
        \node [draw,fill,circle,inner sep = 0pt, minimum size = .25cm,label=above:{\small$\ol{0}$}] (C1) at (-.8,3.35) {};

    \draw (N1) -- (M2);   
    \draw (M4) -- (N4);
    \draw (M6) -- (N5);
    \draw (M5) -- (N4);   
    \draw (N1) -- (N2);
    \draw (N3) -- (N4);
    \draw (C1) -- (M1);
        \end{tikzpicture}
       \captionsetup{width=1.0\linewidth}
  \captionof{figure}{The associated birooted forest.}
  \label{fig:birooted}
 \end{center}
 
 %%%%%%%%%%%%%%%%%%%%%%%%%%%%%%%%%%%%%%%%%%%%%%%%%%%%%%%%%%%%%
 \vfill\null
\columnbreak
%%%%%%%%%%%%%%%%%%%%%%%%%%%%%%%%%%%%%%%%%%%%%%%%%%%%%%%%%%%%%%%
\renewcommand{\thefigure}{3C}        
 \begin{center}
 
        \begin{tikzpicture}[scale=.6]
    %\centering
        \node [draw,fill,circle,inner sep = 0pt, minimum size = .15cm,label=below:{\small1}] (M1) at (-.5,0) {};
        \node [draw,fill,circle,inner sep = 0pt, minimum size = .15cm,label=below:{\small2}] (M2) at (.5,0) {};
        \node [draw,fill,circle,inner sep = 0pt, minimum size = .15cm,label=below:{\small3}] (M3) at (1.5,0) {};
        \node [draw,fill,circle,inner sep = 0pt, minimum size = .15cm,label=below:{\small4}] (M4) at (2.5,0) {};
        \node [draw,fill,circle,inner sep = 0pt, minimum size = .15cm,label=above:{\small$\ol{1}$}] (N1) at (0,2) {};
        \node [draw,fill,circle,inner sep = 0pt, minimum size = .15cm,label=above:{\small$\ol{2}$}] (N2) at (1,2) {};
        \node [draw,fill,circle,inner sep = 0pt, minimum size = .15cm,label=above:{\small$\ol{3}$}] (N3) at (2,2) {};
        \node [draw,fill,circle,inner sep = 0pt, minimum size = .15cm,label=above:{\small$\ol{4}$}] (N4) at (3,2) {};
        \node [draw,fill,circle,inner sep = 0pt, minimum size = .15cm,label=below:{\small5}] (M5) at (3.5,0) {};
        \node [draw,fill,circle,inner sep = 0pt, minimum size = .15cm,label=below:{\small6}] (M6) at (4.5,0) {};
        \node [draw,fill,circle,inner sep = 0pt, minimum size = .15cm,label=above:{\small$\ol{5}$}] (N5) at (4,2) {};
        
        %cone points
        \node [inner sep = 0pt, minimum size = .15cm,label=below:{\small\textcolor{white}{0}}] (C0) at (-.8,-1.35) {};
        \node [draw,fill,circle,inner sep = 0pt, minimum size = .15cm,label=above:{\small$\ol{0}$}] (C1) at (-.8,3.35) {};

         \draw [line width=1pt,style=double,thick](M2) -- (N1);
        \draw (M1) -- (C1);   
        \draw (N3) -- (N4);
        \draw (M6) -- (N5);
        \draw (M5) -- (N4);   
        \draw [line width=1pt,style=double,thick] (N1) -- (N2);
        \draw [line width=1pt,style=double,thick] (M4) -- (N4);

        \end{tikzpicture}
      \captionsetup{width=1.0\linewidth}
  \captionof{figure}{The associated 2-weighted forest.}
    \label{fig:2weighted}
  \end{center}

\end{multicols}

\end{figure}

By composing the two maps we created in this section we get a bijection $\phi = \phi_2 \circ \phi_1$ from $\T(\G)$ to $\F(\GG)$:

\begin{corollary}\label{cor:mainbijection}
For any biconed graph $\G$ the function $\phi:\T(\G) \rightarrow \F(\GG)$ defined above is a bijection from the set of spanning trees of $\G$ to the collection of $2$-weighted forests of $\GG$.
\end{corollary}

See Figure~$3$ for an illustration of the correspondence. In the next section we will show that the collection of $2$-weighted forests gives rise to a pure multicomplex, leading to a proof of Stanley's conjecture for matroids arising from biconed graphs.

\section{A multicomplex of 2-weighted forests}\label{sec:multi}

In this section, we show that the collection $\F(\GG)$ of $2$-weighted forests associated to a biconed graph $\G$ provides a pure multicomplex that encodes the $h$-vector of $\G$. In this context, it will be convenient to think of elements of $\F(\GG)$ as monomials.  For this, we associate a variable to each edge of $\GG$ and construct a monomial from a weighed set $(F, \omega)$ with $F \subset \F(\GG)$ by simply raising all elements to the corresponding weighting.  For example the weighted set $\{c,e,f\}$ with $\omega(e) = \omega(f) = 1$ and $\omega(c) = 3$ is labeled $c^3ef$ (see Figure \ref{fig:monomialc3ef}). Via this correspondence, we will refer to the \emph{degree} of an element of $\F(\GG)$, by which we mean the degree of the corresponding monomial. Under this viewpoint, we prove that $\F(\GG)$ is a pure multicomplex and the resulting pure ${\mathcal O}$-sequence is the $h$-vector of $\G$. The first step is to show how the degree of a monomial encodes passivity of the underlying spanning tree.

\begin{example}
\label{multiset}
Let $G$ be $K_4$ with the vertices partitioned by $A = \{1,2\}$, $B = \ol{A} = \{\ol{1},\ol{2}\}$, and $E(G) = \{a,b,c,d,e,f\}$.  The 2-weighted forest in Figure \ref{fig:monomiala2df2} can be thought of as the monomial $a^2df^2$.  In Figure \ref{fig:monomialaf2} we have the monomial $af^2$, obtained by removing an edge and lowering the weight of another edge from  Figure \ref{fig:monomiala2df2}. Note that this results in a monomial that divides the original.
\end{example}

\begin{figure}
%%%%%%%%%%%%%%%%%%%%%%%%%%%%%%%%%%%%%%%%%%%%%%%%%%%%%%%%%%%%%%%%%%%%%%%%%%%%%%%%%%%%%%%%%%%%%%%%%%%%%%%%%%%%%%%%%%%%%%%%%%%%%%%%%%%%%%%%%%%
\begin{multicols}{3}
\begin{center}
    \renewcommand{\thefigure}{4A}\label{fig4}
        \begin{tikzpicture}
    
    \node [draw,fill,circle,inner sep = 0pt, minimum size = .15cm,label=below:{\small1}] (M2) at (0,0) {};
    \node [draw,fill,circle,inner sep = 0pt, minimum size = .15cm,label=below:{\small2}] (B2) at (2,0) {};
    \node [draw,fill,circle,inner sep = 0pt, minimum size = .15cm,label=above:{\small$\ol{1}$}] (M1) at (0,2) {};
    \node [draw,fill,circle,inner sep = 0pt, minimum size = .15cm,label=above:{\small$\ol{2}$}] (B1) at (2,2) {};
    %Edge labels
    %\node [inner sep = 0pt, minimum size = .15cm,label=above:{\small\textcolor{black}{$x_{1\ol{2}}^3$}}] (L0) at (1,2) {};
    \node [label=left:{$c^3$}] (L1) at (1,.9) {};
    \node [label=below:{\small$f$}] (L2) at (1,0) {};
    \node [label=right:{\small$e$}] (L3) at (2,1) {};
    
    \draw[triple={[line width=.5mm,black] in
    [line width=1mm,white] in
    [line width=2mm,black]}] (M1) to (B2);
    \draw (M2) -- (B2);
    \draw (B1) -- (B2);
    
    \end{tikzpicture}
    \captionsetup{width=1.0\linewidth}
  \captionof{figure}{A $2$-weighted forest giving rise to the monomial $c^3 e f$.}
   \label{fig:monomialc3ef}
\end{center}

%%%%%%%%%%%%%%%%%%%%%%%%%%%%%%%%%%%%%%%%%%%%%%%%%%%%%%%%%%%%%%%%%%%%
\vfill\null
\columnbreak
%%%%%%%%%%%%%%%%%%%%%%%%%%%%%%%%%%%%%%%%%%%%%%%%%%%%%%%%%%%%%%%%%%%
\begin{center}
\renewcommand{\thefigure}{4B}

        \begin{tikzpicture}
    
    \node [draw,fill,circle,inner sep = 0pt, minimum size = .15cm,label=below:{\small1}] (M2) at (0,0) {};
    \node [draw,fill,circle,inner sep = 0pt, minimum size = .15cm,label=below:{\small2}] (B2) at (2,0) {};
    \node [draw,fill,circle,inner sep = 0pt, minimum size = .15cm,label=above:{\small$\ol{1}$}] (M1) at (0,2) {};
    \node [draw,fill,circle,inner sep = 0pt, minimum size = .15cm,label=above:{\small$\ol{2}$}] (B1) at (2,2) {};
    %Edge labels
    \node [label=left:{\small$d$}] (L1) at (1,1) {};
    \node [label=below:{\small$f^2$}] (L2) at (1,0) {};
    \node [label=above:{\small$a^2$}] (L3) at (1,2) {};
    
    \draw [line width=1pt,style=double,thick] (M1) -- (B1);
    \draw [line width=1pt,style=double,thick] (M2) -- (B2);
    \draw (B1) -- (M2);
    
    \end{tikzpicture}
    \captionsetup{width=1.0\linewidth}
  \captionof{figure}{{A $2$-weighted forest corresponding to the monomial $a^2 d f^2$.}}
   \label{fig:monomiala2df2}
\end{center}

%%%%%%%%%%%%%%%%%%%%%%%%%%%%%%%%%%%%%%%%%%%%%%%%%%%%%%%%%%%%%%%%%%%% 
\vfill\null
\columnbreak
%%%%%%%%%%%%%%%%%%%%%%%%%%%%%%%%%%%%%%%%%%%%%%%%%%%%%%%%%%%%%%%%%%%%
\begin{center}
\renewcommand{\thefigure}{4C}

        \begin{tikzpicture}
    
    \node [draw,fill,circle,inner sep = 0pt, minimum size = .15cm,label=below:{\small1}] (M2) at (0,0) {};
    \node [draw,fill,circle,inner sep = 0pt, minimum size = .15cm,label=below:{\small2}] (B2) at (2,0) {};
    \node [draw,fill,circle,inner sep = 0pt, minimum size = .15cm,label=above:{\small$\ol{1}$}] (M1) at (0,2) {};
    \node [draw,fill,circle,inner sep = 0pt, minimum size = .15cm,label=above:{\small$\ol{2}$}] (B1) at (2,2) {};
    %Edge labels
    \node [label=below:{\small$f^2$}] (L2) at (1,0) {};
    \node [label=above:{\small$a$}] (L3) at (1,2) {};
    %\node [label=right:{\small$x_{12}^2$}] (L1) at (0,1) {};
    
    \draw [line width=1pt,thick] (M1) -- (B1);
    %\draw (M2) -- (B2);
    \draw [line width=1pt,style=double,thick] (M2) -- (B2);
    
    \end{tikzpicture}
    \captionsetup{width=1.0\linewidth}
  \captionof{figure}{{A $2$-weighted forest with monomial $a f^2$.}}
   \label{fig:monomialaf2}
\end{center}

\end{multicols}
\end{figure}

\begin{lemma} \label{lem:activity}
\label{BandF}
The function $\phi:\T(\G) \rightarrow \F(\GG)$ described in Corollary \ref{cor:mainbijection} maps a spanning tree of $\G$ with $i$ internally passive edges to a monomial of degree $i$.
%equal to the number of $2$-edge-rooted-forests of $G$ with degree $i$, that is, $$\abs{\B(\G)_i}=\abs{\F(\red)_i}$$
\end{lemma}

\begin{proof}
Suppose $T$ is a spanning tree of $\G$.  We use Lemma \ref{lem:active} to identify the internally passive edges of $T$ and analyze how this information is transferred to the edge weights of $\phi(T)$, resulting in a monomial with degree equal to the number of internally passive edges of $T$. It suffices to show that each passive edge of $T$ increases the degree of the monomial by $1$.

From Lemma \ref{lem:active}, recall that $0\ol{0}$, if present, must internally active, and all edges of $T \setminus T_0$ are internally passive. These internally passive edges are the variables of our monomial, corresponding to a subset of the edges of $\GG$.  The remaining edges of $T$ connect $0$ or $\ol{0}$ to connecting vertices of $T$, and we will encode any internal passivity amongst these edges via the extra weighting on the edges of $T \setminus T_0$.

Let $xv$ be a connecting edge in $T$ (so that $x=0$ or $x = \ol{0}$) and let $C_v$ denote its component in $T\setminus T_0$.  We now have two cases to consider. 

First suppose that $v$ is the only connecting vertex in $C_v$. By Lemma \ref{lem:active}, the edge $xv$ is passive if and only if $v$ is not the smallest vertex in $C_v$.  In a component with one connecting vertex (this corresponds to a rooted component which does not contain $\ol 0$ or a birooted component which does contain $\ol 0$), $\phi$ assigns an edge in the component an extra weight exactly when $v$ is not the smallest vertex in $C_v$, thus increasing the degree of the monomial by $1$ if and only if $xv$ is passive.

Next suppose that $C_v$ has two connecting vertices. Again, Lemma \ref{lem:active} tells us that $xv$ is passive. The component $C_v$ receives an extra weight inside $\phi(T)$, by weighting the first edge in the path from $v$ to the other root inside $C_v$, thus increasing the degree of the monomial by $1$ (the component thus ends up with excess weight $2$, with a contribution of $1$ from each of its two connecting vertices).

Components with no connecting vertices must contain $\ol 0$ and must be rooted, but not birooted, and receive no extra weighting under $\phi$.

We have seen that each passive edge of $T \setminus T_0$ contributes an extra weight to some edge of $\phi(T)$, and hence an extra degree to the desired monomial.  These are the only situations under which $\phi$ assigns excess weight to an edges, so this proves the claim.
\end{proof}

Note that a $2$-weighted forest $F$ can naturally be thought of as a multiset on the underlying set of edges of $\GG$, where the number of occurrences of any edge is given by its weighting.  Our next two lemmas show that this collection forms a pure multicomplex.

\begin{lemma}
\label{lem:multi}
For any biconed graph $\G$, the multiset $\F(\GG)$ is a multicomplex on the edges of $\GG$. 
\end{lemma}

\begin{proof}
We must show that the set of $2$-weighted forests is closed under taking subsets.  To this end, let $(F,\omega) \in \F(\GG)$ be a $2$-weighted forest, and let $e \in F$ be any edge. We define a new multiset  $(F^\prime, \omega^\prime)$ obtained by removing one occurrence of $e$, and want to show that it produces a valid $2$-weighted forest. We have two cases to consider, depending on the weight $\omega(e)$.

First suppose $\omega(e) \geq 2$.  We then  define $(F^\prime, \omega^\prime)$ as $F^\prime = F$ and $\omega^\prime(e) = \omega(e)-1$ (with the weights equal everywhere else).   We check the conditions given in Definition \ref{def:2erf}. Condition $(C0)$ is clearly satisfied since we have the same underlying set $F$ of edges. Similarly, conditions $(C1)$ and $(C2)$ are satisfied since $\omega^\prime(x) \leq \omega(x)$ for all edges $x \in F$. For $(C3)$, we assume that $(F,\omega)$ has a component $C$ with excess weight $2$ (since otherwise there is nothing to check). If $e \in C$ then $(F^\prime, \omega^\prime)$ now has no component with excess weight $2$. If $e \notin C$ we see that $C$ is not affected and hence still satisfies the condition.

Next we suppose $\omega(e) = 1$. We define $(F^\prime, \omega^\prime)$ as $F^\prime = F \backslash e$ and $\omega^\prime = \omega|_{F^\prime}$. In this case, $(C0)$ is satisfied since $F^\prime \subset F$ still cannot have any cycles.  Conditions $(C1)$ and $(C2)$ are again satisfied since $\omega^\prime(x) \leq \omega(x)$ for all edges $x \in F^\prime$. For condition $(C3)$  we again can assume that $(F,\omega)$ has a component $C$ with excess weight $2$. If $e$ is one of the edges in the unique path containing the weighted edges of $C$, then removing $e$ results in an $(F^\prime, \omega^\prime)$ with no component with excess weight $2$. Otherwise, the path between the weighted edges is unaltered. Either way, condition $(C3)$ is satisfied.  We conclude that $(F^\prime, \omega^\prime) \in \F(\GG)$, and the result follows.
\end{proof}

%%% OLD PROOF 
 
% \begin{proof}
% We need to show that for any $\Fm \in \F(\GG)$ and any $e \in \Fm$, the multiset $\Fm \setminus e$ is again in $\F(\GG)$. This is where our definition of $2$-weighted forests will come in handy. From Definition~\ref{def:2erf}, the only condition that needs to be carefully checked after deleting an edge is $(C3)$. 

% In the case that deletion of $e$ results in either no change to the component with excess weight $2$ or
% results in a forest with no components with excess weight $2$ (as is illustrated in moving between the graphs in Figure $3$B to $3$C), the condition $(C3)$ is obviously satisfied. In the remaining case, that is, when $e$ is not part of a path between the rooted edges, it does not affect the crossing edges so $(C3)$ is again satisfied.
% \end{proof}

%Now we show that the maximal elements in $\F(\GG)$ are all the same size.

\begin{lemma} \label{lem:pure}
For any biconed graph $\G$, the multicomplex $\F(\GG)$ is pure.
\end{lemma}
\begin{proof}
%Recall that a coloop of a graph is an edge that is contained in every spanning tress of the graph. 
We must show that all maximal elements (under inclusion) of $\F(\GG)$ have the same degree. Recall from Lemma \ref{lem:activity} that the degree of a maximal element $(F,\omega)$ is given by the number of internally passive edges in the spanning tree $T$ where $\phi(T) = (F, \omega)$.  We start by identifying which edges are always active. Any bridge (an edge whose removal increases the number of connected components) in $\G$ must be contained in every spanning tree and furthermore must be internally active.  Bridges in $\G$ must have either $0$ or $\ol{0}$ as an endpoint (and $0\ol{0}$ is a bridge if and only if $\GG$ has no crossing edges). Any spanning forest of $\GG$ will contain a singular component (a component consisting of a single vertex) for each bridge of $\G$ that is not $0\ol{0}$.

Now let $d$ be the largest degree of any monomial in $\F(\GG)$. Note that a $2$-weighted forest of degree $d$ corresponds via $\phi$ to a spanning tree of $\G$ with the property that every non-bridge edge is passive.   From the above discussion, and translating the description of passive edges to $\F(\GG)$ via Lemma \ref{lem:active}, we see that $(F,\omega) \in \F(\GG)$ has degree $d$ if and only if 

\begin{itemize}
    \item all nonsingular components of $(F,\omega)$ have excess weight at least $1$,
    \item if $\GG$ has at least one crossing edge then $(F, \omega)$ has exactly one nonsingular component with excess weight $2$, and 
    \item singular components are endpoints of bridges in $\G$.
    \end{itemize}
To prove the lemma, it suffices to show that for any $(F,\omega) \in \F(\GG)$ that does not satisfy these properties, there is some $(F', \omega') \in \F(\GG)$ of higher degree such that $F \subseteq F'$ and $\omega(e) \leq \omega'(e)$ for any $e$ where $\omega'(e)$ is defined. To find such a $2$-weighted forest $(F', \omega')$, we identify an edge which may be added to $F$ or have its weight increased to get another valid $2$-weighted forest. We break it down into several cases.  For each case the conditions $(C0)$, $(C1)$, and $(C2)$ of Definition \ref{def:2erf} will be clearly satisfied by $(F', \omega')$, and only $(C3)$ will need  some explanation. 

First note that if $(F,\omega)$ contains any non-singular component which is of excess weight 0, we may increase the weight of any edge in the component to get our desired $(F', \omega')$. If $F$ contains a singular component $v$ that is not $\ol{0}$ and is not the endpoint of a bridge in $\G$, we may add any edge of $\GG$ having $v$ as an endpoint to $F$. This does not create a cycle as $v$ was isolated.

The remaining case is when all non-singular components of $(F, \omega)$ have excess weight $1$, but $(F, \omega)$ does not have a component with excess weight $2$. In this case, $\GG$ must contain a crossing edge, or else $(F, \omega)$ would already be of highest degree. 
If $(F,\omega)$ itself does not contain any crossing edges, we simply add any crossing edge of $\GG$ to get $(F', \omega')$. This creates a component with excess weight $2$ which has the new edge as the only crossing edge in the path between the weighted edges, hence satisfying $(C3)$. If $(F,\omega)$ contains a crossing edge that has weight 2, we increase its weight to 3 to get $(F', \omega')$. This edge is then the only edge in the path containing all weighted edges in the component, and $(F', \omega')$ satisfies $(C3)$.  

The last subcase to consider is if $(F, \omega)$ contains at least one crossing edge, and all such crossing edges are of weight $1$. Here we must take care to ensure that the parity condition in $(C3)$ is satisfied. We pick any component with a crossing edge and find the nearest (shortest path length) crossing edge to the component's edge of weight $2$ (or to $\ol{0}$ if there is no edge of weight $2$), picking arbitrarily if there is a tie. We weight this edge to get $(F', \omega')$. This ensures that there is exactly one crossing edge in the path containing both weighted edges (or containing the weighted edge and $\ol{0}$) in the component with excess weight $2$. Indeed, if there was another crossing edge in the path, it would have been closer than the edge selected. This exhausts all cases and the result follows.
\end{proof}

For a biconed graph $\G$, we let $h_i$ denote the $i$th entry of the $h$-vector of its underlying graphic matroid. By Lemma \ref{h-vector} and Lemma \ref{BandF}, we have that  $h_i$ is given by the number of monomials in $\F(\red)$ of degree $i$.  From Lemma \ref{lem:multi} and Lemma \ref{lem:pure}, we that $\F(\red)$ is a pure multicomplex. This proves our main result:
\begin{corollary}\label{cor:Stanley}
Stanley's $h$-vector conjecture holds for graphic matroids of biconed graphs.
\end{corollary}

\section{Worked Example}\label{sec:Example}
Consider the graph $\G$ depicted in Figure \ref{fig:biconed}, obtained by biconing the graph $G$ in Figure \ref{fig:precone}.  The spanning tree $T_0$ and the graph $\GG$ are depicted in Figures \ref{fig:lexfirst} and \ref{fig:reduced}.  We see that $\G$ has 11 spanning trees (see Figure \ref{fig:trees}) and has $h$-vector given by $(1,2,3,3,2)$.  The collection of birooted forests is depicted in Figure \ref{fig:biroot}, and the complex of $2$-weighted forests of $\GG$ is seen in Figure \ref{fig:multicomp}. Note that the degree sequence of this multicomplex agrees with the $h$-vector.  The poset structure in these figures is determined by the divisibility of monomials in Figure \ref{fig:multicomp}, although one can check that the poset in Figure \ref{fig:trees} is the \emph{internal order} on trees (see Section \ref{sec:Klee-Samper} for more discussion).

\begin{figure}[H]
\begin{multicols}{4}
\renewcommand{\thefigure}{5A}
 \begin{center}
 
          \begin{tikzpicture}[scale=.6]
        \node [draw,fill,circle,inner sep = 0pt, minimum size = .15cm,label=below:{\small1}] (M1) at (-.5,0) {};
        \node [draw,fill,circle,inner sep = 0pt, minimum size = .15cm,label=below:{\small2}] (M2) at (.5,0) {};
        \node [draw,fill,circle,inner sep = 0pt, minimum size = .15cm,label=above:{\small$\ol{1}$}] (N1) at (0,2) {};
        
        %cone points
        %\node [draw,fill,circle,inner sep = 0pt, minimum size = .15cm,label=below:{\small0}] (C0) at (-1,-1.35) {};
        %\node [draw,fill,circle,inner sep = 0pt, minimum size = .15cm,label=above:{\small$\ol{0}$}] (C1) at (-1,3.35) {};
        
    %\draw (C0) -- (C1);   
    %\draw (C0) -- (M1);
    %\draw (C0) -- (M2);
    %\draw (C1) -- (N1);
     \draw (M2) -- (N1);
    %\draw (C1) -- (M1);

    \end{tikzpicture}
    \captionsetup{width=1.0\linewidth}
  \captionof{figure}{A graph $G$}
 \label{fig:precone}
 \end{center}

%%%%%%%%%%%%%%%%%%%%%%%%%%%%%%%%%%%%%%%%%%%%%%%%%%%%%%%%%%%%%%%%%%%
\vfill\null
\columnbreak
%%%%%%%%%%%%%%%%%%%%%%%%%%%%%%%%%%%%%%%%%%%%%%%%%%%%%%%%%%%%%%%%%%%
\renewcommand{\thefigure}{5B}
 \begin{center}

      \begin{tikzpicture}[scale=.6]
        \node [draw,fill,circle,inner sep = 0pt, minimum size = .15cm,label=below:{\small1}] (M1) at (-.5,0) {};
        \node [draw,fill,circle,inner sep = 0pt, minimum size = .15cm,label=below:{\small2}] (M2) at (.5,0) {};
        \node [draw,fill,circle,inner sep = 0pt, minimum size = .15cm,label=above:{\small$\ol{1}$}] (N1) at (0,2) {};
        
        %cone points
        \node [draw,fill,circle,inner sep = 0pt, minimum size = .15cm,label=below:{\small0}] (C0) at (-1,-1.35) {};
        \node [draw,fill,circle,inner sep = 0pt, minimum size = .15cm,label=above:{\small$\ol{0}$}] (C1) at (-1,3.35) {};
        
    \draw (C0) -- (C1);   
    \draw (C0) -- (M1);
    \draw (C0) -- (M2);
    \draw (C1) -- (N1);
     \draw (M2) -- (N1);
    \draw (C1) -- (M1);

    \end{tikzpicture}
       \captionsetup{width=1.0\linewidth}
  \captionof{figure}{The biconed $\G$ with $A=\{1,2\}$ and $B=\{1,\ol{1}\}$}
 \label{fig:biconed}
 \end{center}
 
 %%%%%%%%%%%%%%%%%%%%%%%%%%%%%%%%%%%%%%%%%%%%%%%%%%%%%%%%%%%%%
 \vfill\null
\columnbreak
%%%%%%%%%%%%%%%%%%%%%%%%%%%%%%%%%%%%%%%%%%%%%%%%%%%%%%%%%%%%%%%
\renewcommand{\thefigure}{5C}        
 \begin{center}

        \begin{tikzpicture}[scale=.6]
        \node [draw,fill,circle,inner sep = 0pt, minimum size = .15cm,label=below:{\small1}] (M1) at (-.5,0) {};
        \node [draw,fill,circle,inner sep = 0pt, minimum size = .15cm,label=below:{\small2}] (M2) at (.5,0) {};
        \node [draw,fill,circle,inner sep = 0pt, minimum size = .15cm,label=above:{\small$\ol{1}$}] (N1) at (0,2) {};
        
        %cone points
        \node [draw,fill,circle,inner sep = 0pt, minimum size = .15cm,label=below:{\small0}] (C0) at (-1,-1.35) {};
        \node [draw,fill,circle,inner sep = 0pt, minimum size = .15cm,label=above:{\small$\ol{0}$}] (C1) at (-1,3.35) {};
        
    \draw (C0) -- (C1);   
    \draw (C0) -- (M1);
    \draw (C0) -- (M2);
    \draw (C1) -- (N1);

    \end{tikzpicture}
      \captionsetup{width=1.0\linewidth}
  \captionof{figure}{The lex-first spanning tree $T_0$}
   \label{fig:lexfirst}
  \end{center}
  
%%%%%%%%%%%%%%%%%%%%%%%%%%%%%%%%%%%%%%%%%%%%%%%%%%%%%%%%%%%%%
 \vfill\null
\columnbreak
%%%%%%%%%%%%%%%%%%%%%%%%%%%%%%%%%%%%%%%%%%%%%%%%%%%%%%%%%%%%%%%
\renewcommand{\thefigure}{5D}        
 \begin{center}
           \begin{tikzpicture}[scale=.6]
        \node [draw,fill,circle,inner sep = 0pt, minimum size = .15cm,label=below:{\small1}] (M1) at (-.5,0) {};
        \node [draw,fill,circle,inner sep = 0pt, minimum size = .15cm,label=below:{\small2}] (M2) at (.5,0) {};
        \node [draw,fill,circle,inner sep = 0pt, minimum size = .15cm,label=above:{\small$\ol{1}$}] (N1) at (0,2) {};
        
        %cone points
        %\node [draw,fill,circle,inner sep = 0pt, minimum size = .15cm,label=below:{\small0}] (C0) at (-1,-1.35) {};
        \node [draw,fill,circle,inner sep = 0pt, minimum size = .15cm,label=above:{\small$\ol{0}$}] (C1) at (-1,3.35) {};
        
    %\draw (C0) -- (C1);   
    %\draw (C0) -- (M1);
    %\draw (C0) -- (M2);
    %\draw (C1) -- (N1);
     \draw (M2) -- (N1);
    \draw (C1) -- (M1);

    \end{tikzpicture}
      \captionsetup{width=1.0\linewidth}
  \captionof{figure}{The `reduced ' graph $\G_{\text{red}}$}
   \label{fig:reduced}
  \end{center}

\end{multicols}
\end{figure}

\begin{figure}[H]
\renewcommand{\thefigure}{6A}
 \begin{center}

   \scalebox{.75}{         \begin{tikzpicture}[scale=.6]
        \node (1) {        \begin{tikzpicture}[scale=.6,anchor=center]
        \node [draw,fill,circle,inner sep = 0pt, minimum size = .15cm,label=below:{\small1}] (M1) at (-.5,0) {};
        \node [draw,fill,circle,inner sep = 0pt, minimum size = .15cm,label=below:{\small2}] (M2) at (.5,0) {};
        \node [draw,fill,circle,inner sep = 0pt, minimum size = .15cm,label=above:{\small$\ol{1}$}] (N1) at (0,2) {};
        
        %cone points
        \node [draw,fill,circle,inner sep = 0pt, minimum size = .15cm,label=below:{\small0}] (C0) at (-1,-1.35) {};
        \node [draw,fill,circle,inner sep = 0pt, minimum size = .15cm,label=above:{\small$\ol{0}$}] (C1) at (-1,3.35) {};
        
    \draw (C0) -- (C1);   
    \draw (C0) -- (M1);
    \draw (C0) -- (M2);
    \draw (C1) -- (N1);

    \end{tikzpicture}};
        \node[xshift=70, yshift=70] (y) {        \begin{tikzpicture}[scale=.6]
        \node [draw,fill,circle,inner sep = 0pt, minimum size = .15cm,label=below:{\small1}] (M1) at (-.5,0) {};
        \node [draw,fill,circle,inner sep = 0pt, minimum size = .15cm,label=below:{\small2}] (M2) at (.5,0) {};
        \node [draw,fill,circle,inner sep = 0pt, minimum size = .15cm,label=above:{\small$\ol{1}$}] (N1) at (0,2) {};
        
        %cone points
        \node [draw,fill,circle,inner sep = 0pt, minimum size = .15cm,label=below:{\small0}] (C0) at (-1,-1.35) {};
        \node [draw,fill,circle,inner sep = 0pt, minimum size = .15cm,label=above:{\small$\ol{0}$}] (C1) at (-1,3.35) {};
        
    \draw (C0) -- (C1);   
    \draw (C0) -- (M1);
    \draw (C0) -- (M2);
    \draw (N1) -- (M2);

    \end{tikzpicture}};
        \node[xshift=-70, yshift=70] (x) {        \begin{tikzpicture}[scale=.6]
        \node [draw,fill,circle,inner sep = 0pt, minimum size = .15cm,label=below:{\small1}] (M1) at (-.5,0) {};
        \node [draw,fill,circle,inner sep = 0pt, minimum size = .15cm,label=below:{\small2}] (M2) at (.5,0) {};
        \node [draw,fill,circle,inner sep = 0pt, minimum size = .15cm,label=above:{\small$\ol{1}$}] (N1) at (0,2) {};
        
        %cone points
        \node [draw,fill,circle,inner sep = 0pt, minimum size = .15cm,label=below:{\small0}] (C0) at (-1,-1.35) {};
        \node [draw,fill,circle,inner sep = 0pt, minimum size = .15cm,label=above:{\small$\ol{0}$}] (C1) at (-1,3.35) {};
        
    \draw (C0) -- (C1);   
    \draw (C1) -- (M1);
    \draw (C0) -- (M2);
    \draw (C1) -- (N1);

    \end{tikzpicture}};
        
        \node[xshift=-140, yshift=140] (x2) {        \begin{tikzpicture}[scale=.6]
        \node [draw,fill,circle,inner sep = 0pt, minimum size = .15cm,label=below:{\small1}] (M1) at (-.5,0) {};
        \node [draw,fill,circle,inner sep = 0pt, minimum size = .15cm,label=below:{\small2}] (M2) at (.5,0) {};
        \node [draw,fill,circle,inner sep = 0pt, minimum size = .15cm,label=above:{\small$\ol{1}$}] (N1) at (0,2) {};
        
        %cone points
        \node [draw,fill,circle,inner sep = 0pt, minimum size = .15cm,label=below:{\small0}] (C0) at (-1,-1.35) {};
        \node [draw,fill,circle,inner sep = 0pt, minimum size = .15cm,label=above:{\small$\ol{0}$}] (C1) at (-1,3.35) {};
        
    \draw (C0) -- (M1);   
    \draw (C1) -- (M1);
    \draw (C0) -- (M2);
    \draw (C1) -- (N1);

    \end{tikzpicture}};
        \node[xshift=0, yshift=140] (xy) {        \begin{tikzpicture}[scale=.6]
        \node [draw,fill,circle,inner sep = 0pt, minimum size = .15cm,label=below:{\small1}] (M1) at (-.5,0) {};
        \node [draw,fill,circle,inner sep = 0pt, minimum size = .15cm,label=below:{\small2}] (M2) at (.5,0) {};
        \node [draw,fill,circle,inner sep = 0pt, minimum size = .15cm,label=above:{\small$\ol{1}$}] (N1) at (0,2) {};
        
        %cone points
        \node [draw,fill,circle,inner sep = 0pt, minimum size = .15cm,label=below:{\small0}] (C0) at (-1,-1.35) {};
        \node [draw,fill,circle,inner sep = 0pt, minimum size = .15cm,label=above:{\small$\ol{0}$}] (C1) at (-1,3.35) {};
        
    \draw (C0) -- (C1);   
    \draw (N1) -- (M2);
    \draw (C0) -- (M2);
    \draw (C1) -- (M1);

    \end{tikzpicture}};
        \node[xshift=140, yshift=140] (y2) {        \begin{tikzpicture}[scale=.6]
        \node [draw,fill,circle,inner sep = 0pt, minimum size = .15cm,label=below:{\small1}] (M1) at (-.5,0) {};
        \node [draw,fill,circle,inner sep = 0pt, minimum size = .15cm,label=below:{\small2}] (M2) at (.5,0) {};
        \node [draw,fill,circle,inner sep = 0pt, minimum size = .15cm,label=above:{\small$\ol{1}$}] (N1) at (0,2) {};
        
        %cone points
        \node [draw,fill,circle,inner sep = 0pt, minimum size = .15cm,label=below:{\small0}] (C0) at (-1,-1.35) {};
        \node [draw,fill,circle,inner sep = 0pt, minimum size = .15cm,label=above:{\small$\ol{0}$}] (C1) at (-1,3.35) {};
        
    \draw (C0) -- (C1);   
    \draw (C0) -- (M1);
    \draw (N1) -- (M2);
    \draw (C1) -- (N1);

    \end{tikzpicture}};
        
        \node[xshift=-70, yshift=210] (x2y) {        \begin{tikzpicture}[scale=.6]
        \node [draw,fill,circle,inner sep = 0pt, minimum size = .15cm,label=below:{\small1}] (M1) at (-.5,0) {};
        \node [draw,fill,circle,inner sep = 0pt, minimum size = .15cm,label=below:{\small2}] (M2) at (.5,0) {};
        \node [draw,fill,circle,inner sep = 0pt, minimum size = .15cm,label=above:{\small$\ol{1}$}] (N1) at (0,2) {};
        
        %cone points
        \node [draw,fill,circle,inner sep = 0pt, minimum size = .15cm,label=below:{\small0}] (C0) at (-1,-1.35) {};
        \node [draw,fill,circle,inner sep = 0pt, minimum size = .15cm,label=above:{\small$\ol{0}$}] (C1) at (-1,3.35) {};
        
    \draw (M2) -- (N1);   
    \draw (C0) -- (M1);
    \draw (C0) -- (M2);
    \draw (C1) -- (M1);

    \end{tikzpicture}};
        \node[xshift=70, yshift=210] (xy2) {        \begin{tikzpicture}[scale=.6]
        \node [draw,fill,circle,inner sep = 0pt, minimum size = .15cm,label=below:{\small1}] (M1) at (-.5,0) {};
        \node [draw,fill,circle,inner sep = 0pt, minimum size = .15cm,label=below:{\small2}] (M2) at (.5,0) {};
        \node [draw,fill,circle,inner sep = 0pt, minimum size = .15cm,label=above:{\small$\ol{1}$}] (N1) at (0,2) {};
        
        %cone points
        \node [draw,fill,circle,inner sep = 0pt, minimum size = .15cm,label=below:{\small0}] (C0) at (-1,-1.35) {};
        \node [draw,fill,circle,inner sep = 0pt, minimum size = .15cm,label=above:{\small$\ol{0}$}] (C1) at (-1,3.35) {};
        
    \draw (C0) -- (C1);   
    \draw (C1) -- (M1);
    \draw (N1) -- (M2);
    \draw (C1) -- (N1);

    \end{tikzpicture}};
        \node[xshift=210, yshift=210] (y3) {        \begin{tikzpicture}[scale=.6]
        \node [draw,fill,circle,inner sep = 0pt, minimum size = .15cm,label=below:{\small1}] (M1) at (-.5,0) {};
        \node [draw,fill,circle,inner sep = 0pt, minimum size = .15cm,label=below:{\small2}] (M2) at (.5,0) {};
        \node [draw,fill,circle,inner sep = 0pt, minimum size = .15cm,label=above:{\small$\ol{1}$}] (N1) at (0,2) {};
        
        %cone points
        \node [draw,fill,circle,inner sep = 0pt, minimum size = .15cm,label=below:{\small0}] (C0) at (-1,-1.35) {};
        \node [draw,fill,circle,inner sep = 0pt, minimum size = .15cm,label=above:{\small$\ol{0}$}] (C1) at (-1,3.35) {};
        
    \draw (N1) -- (M2);   
    \draw (C0) -- (M1);
    \draw (C0) -- (M2);
    \draw (C1) -- (N1);

    \end{tikzpicture}};
        
        \node[xshift=0, yshift=280] (x2y2) {        \begin{tikzpicture}[scale=.6]
        \node [draw,fill,circle,inner sep = 0pt, minimum size = .15cm,label=below:{\small1}] (M1) at (-.5,0) {};
        \node [draw,fill,circle,inner sep = 0pt, minimum size = .15cm,label=below:{\small2}] (M2) at (.5,0) {};
        \node [draw,fill,circle,inner sep = 0pt, minimum size = .15cm,label=above:{\small$\ol{1}$}] (N1) at (0,2) {};
        
        %cone points
        \node [draw,fill,circle,inner sep = 0pt, minimum size = .15cm,label=below:{\small0}] (C0) at (-1,-1.35) {};
        \node [draw,fill,circle,inner sep = 0pt, minimum size = .15cm,label=above:{\small$\ol{0}$}] (C1) at (-1,3.35) {};
        
    \draw (C1) -- (M1);   
    \draw (C0) -- (M1);
    \draw (N1) -- (M2);
    \draw (C1) -- (N1);

    \end{tikzpicture}};
        \node[xshift=140, yshift=280] (xy3) {        \begin{tikzpicture}[scale=.6]
        \node [draw,fill,circle,inner sep = 0pt, minimum size = .15cm,label=below:{\small1}] (M1) at (-.5,0) {};
        \node [draw,fill,circle,inner sep = 0pt, minimum size = .15cm,label=below:{\small2}] (M2) at (.5,0) {};
        \node [draw,fill,circle,inner sep = 0pt, minimum size = .15cm,label=above:{\small$\ol{1}$}] (N1) at (0,2) {};
        
        %cone points
        \node [draw,fill,circle,inner sep = 0pt, minimum size = .15cm,label=below:{\small0}] (C0) at (-1,-1.35) {};
        \node [draw,fill,circle,inner sep = 0pt, minimum size = .15cm,label=above:{\small$\ol{0}$}] (C1) at (-1,3.35) {};
        
    \draw (C1) -- (M1);   
    \draw (N1) -- (M2);
    \draw (C0) -- (M2);
    \draw (C1) -- (N1);

    \end{tikzpicture}};
        
        \draw[thick] (1) -- (x);
        \draw[thick] (1) -- (y);
        
        \draw[thick] (x) -- (x2);
        \draw[thick] (x) -- (xy);
        \draw[thick] (y) -- (xy);
        \draw[thick] (y) -- (y2);
        
        \draw[thick] (x2) -- (x2y);
        \draw[thick] (xy) -- (x2y);
        \draw[thick] (xy) -- (xy2);
        \draw[thick] (y2) -- (y3);
        \draw[thick] (xy2) -- (y2);
        
        \draw[thick] (x2y) -- (x2y2);
        \draw[thick] (xy2) -- (x2y2);
        \draw[thick] (xy2) -- (xy3);
        \draw[thick] (y3) -- (xy3);
    \end{tikzpicture}}
    \captionsetup{width=1.0\linewidth}
  \captionof{figure}{Spanning trees of $\G$}
 \label{fig:trees}
 \end{center}
\end{figure}

\begin{figure}[H]    
\renewcommand{\thefigure}{6B}
 \begin{center}
 
 \scalebox{.75}{        \begin{tikzpicture}[scale=.6]
        \node (1) {        \begin{tikzpicture}[scale=.6,anchor=center]
        \node [draw,fill,circle,inner sep = 0pt, minimum size = .25cm,label=below:{\small1}] (M1) at (-.5,0) {};
        \node [draw,fill,circle,inner sep = 0pt, minimum size = .25cm,label=below:{\small2}] (M2) at (.5,0) {};
        \node [draw,fill,circle,inner sep = 0pt, minimum size = .25cm,label=above:{\small$\ol{1}$}] (N1) at (0,2) {};
        
        %cone points
        %\node [draw,fill,circle,inner sep = 0pt, minimum size = .15cm,label=below:{\small0}] (C0) at (-1,-1.35) {};
        \node [draw,fill,circle,inner sep = 0pt, minimum size = .25cm,label=above:{\small$\ol{0}$}] (C1) at (-1,3.35) {};
        
    %\draw (C0) -- (C1);   
    %\draw (C0) -- (M1);
    %\draw (C0) -- (M2);
    %\draw (C1) -- (N1);

    \end{tikzpicture}};
        \node[xshift=70, yshift=70] (y) {        \begin{tikzpicture}[scale=.6]
        \node [draw,fill,circle,inner sep = 0pt, minimum size = .25cm,label=below:{\small1}] (M1) at (-.5,0) {};
        \node [draw,fill,circle,inner sep = 0pt, minimum size = .25cm,label=below:{\small2}] (M2) at (.5,0) {};
        \node [draw,fill,circle,inner sep = 0pt, minimum size = .15cm,label=above:{\small$\ol{1}$}] (N1) at (0,2) {};
        
        %cone points
        %\node [draw,fill,circle,inner sep = 0pt, minimum size = .15cm,label=below:{\small0}] (C0) at (-1,-1.35) {};
        \node [draw,fill,circle,inner sep = 0pt, minimum size = .25cm,label=above:{\small$\ol{0}$}] (C1) at (-1,3.35) {};
        
    %\draw (C0) -- (C1);   
    %\draw (C0) -- (M1);
    %\draw (C0) -- (M2);
    \draw (N1) -- (M2);

    \end{tikzpicture}};
        \node[xshift=-70, yshift=70] (x) {        \begin{tikzpicture}[scale=.6]
        \node [draw,fill,circle,inner sep = 0pt, minimum size = .15cm,label=below:{\small1}] (M1) at (-.5,0) {};
        \node [draw,fill,circle,inner sep = 0pt, minimum size = .25cm,label=below:{\small2}] (M2) at (.5,0) {};
        \node [draw,fill,circle,inner sep = 0pt, minimum size = .25cm,label=above:{\small$\ol{1}$}] (N1) at (0,2) {};
        
        %cone points
        %\node [draw,fill,circle,inner sep = 0pt, minimum size = .15cm,label=below:{\small0}] (C0) at (-1,-1.35) {};
        \node [draw,fill,circle,inner sep = 0pt, minimum size = .25cm,label=above:{\small$\ol{0}$}] (C1) at (-1,3.35) {};
        
    %\draw (C0) -- (C1);   
    \draw (C1) -- (M1);
    %\draw (C0) -- (M2);
    %\draw (C1) -- (N1);

    \end{tikzpicture}};
        
        \node[xshift=-140, yshift=140] (x2) {        \begin{tikzpicture}[scale=.6]
        \node [draw,fill,circle,inner sep = 0pt, minimum size = .25cm,label=below:{\small1}] (M1) at (-.5,0) {};
        \node [draw,fill,circle,inner sep = 0pt, minimum size = .25cm,label=below:{\small2}] (M2) at (.5,0) {};
        \node [draw,fill,circle,inner sep = 0pt, minimum size = .25cm,label=above:{\small$\ol{1}$}] (N1) at (0,2) {};
        
        %cone points
        %\node [draw,fill,circle,inner sep = 0pt, minimum size = .15cm,label=below:{\small0}] (C0) at (-1,-1.35) {};
        \node [draw,fill,circle,inner sep = 0pt, minimum size = .25cm,label=above:{\small$\ol{0}$}] (C1) at (-1,3.35) {};
        
    %\draw (C0) -- (M2);   
    \draw (C1) -- (M1);
    %\draw (C0) -- (M2);
    %\draw (C1) -- (N1);

    \end{tikzpicture}};
        \node[xshift=0, yshift=140] (xy) {        \begin{tikzpicture}[scale=.6]
        \node [draw,fill,circle,inner sep = 0pt, minimum size = .15cm,label=below:{\small1}] (M1) at (-.5,0) {};
        \node [draw,fill,circle,inner sep = 0pt, minimum size = .25cm,label=below:{\small2}] (M2) at (.5,0) {};
        \node [draw,fill,circle,inner sep = 0pt, minimum size = .15cm,label=above:{\small$\ol{1}$}] (N1) at (0,2) {};
        
        %cone points
        %\node [draw,fill,circle,inner sep = 0pt, minimum size = .15cm,label=below:{\small0}] (C0) at (-1,-1.35) {};
        \node [draw,fill,circle,inner sep = 0pt, minimum size = .25cm,label=above:{\small$\ol{0}$}] (C1) at (-1,3.35) {};
        
    %\draw (C0) -- (C1);   
    \draw (N1) -- (M2);
    %\draw (C0) -- (M2);
    \draw (C1) -- (M1);

    \end{tikzpicture}};
        \node[xshift=140, yshift=140] (y2) {        \begin{tikzpicture}[scale=.6]
        \node [draw,fill,circle,inner sep = 0pt, minimum size = .25cm,label=below:{\small1}] (M1) at (-.5,0) {};
        \node [draw,fill,circle,inner sep = 0pt, minimum size = .15cm,label=below:{\small2}] (M2) at (.5,0) {};
        \node [draw,fill,circle,inner sep = 0pt, minimum size = .25cm,label=above:{\small$\ol{1}$}] (N1) at (0,2) {};
        
        %cone points
        %\node [draw,fill,circle,inner sep = 0pt, minimum size = .15cm,label=below:{\small0}] (C0) at (-1,-1.35) {};
        \node [draw,fill,circle,inner sep = 0pt, minimum size = .25cm,label=above:{\small$\ol{0}$}] (C1) at (-1,3.35) {};
        
    %\draw (C0) -- (C1);   
    %\draw (C0) -- (M1);
    \draw (N1) -- (M2);
    %\draw (C1) -- (N1);

    \end{tikzpicture}};
        
        \node[xshift=-70, yshift=210] (x2y) {        \begin{tikzpicture}[scale=.6]
        \node [draw,fill,circle,inner sep = 0pt, minimum size = .25cm,label=below:{\small1}] (M1) at (-.5,0) {};
        \node [draw,fill,circle,inner sep = 0pt, minimum size = .25cm,label=below:{\small2}] (M2) at (.5,0) {};
        \node [draw,fill,circle,inner sep = 0pt, minimum size = .15cm,label=above:{\small$\ol{1}$}] (N1) at (0,2) {};
        
        %cone points
        %\node [draw,fill,circle,inner sep = 0pt, minimum size = .15cm,label=below:{\small0}] (C0) at (-1,-1.35) {};
        \node [draw,fill,circle,inner sep = 0pt, minimum size = .25cm,label=above:{\small$\ol{0}$}] (C1) at (-1,3.35) {};
        
    %\draw (C1) -- (N1);   
    %\draw (C0) -- (M1);
    \draw (N1) -- (M2);
    \draw (C1) -- (M1);

    \end{tikzpicture}};
        \node[xshift=70, yshift=210] (xy2) {        \begin{tikzpicture}[scale=.6]
        \node [draw,fill,circle,inner sep = 0pt, minimum size = .15cm,label=below:{\small1}] (M1) at (-.5,0) {};
        \node [draw,fill,circle,inner sep = 0pt, minimum size = .15cm,label=below:{\small2}] (M2) at (.5,0) {};
        \node [draw,fill,circle,inner sep = 0pt, minimum size = .25cm,label=above:{\small$\ol{1}$}] (N1) at (0,2) {};
        
        %cone points
        %\node [draw,fill,circle,inner sep = 0pt, minimum size = .15cm,label=below:{\small0}] (C0) at (-1,-1.35) {};
        \node [draw,fill,circle,inner sep = 0pt, minimum size = .25cm,label=above:{\small$\ol{0}$}] (C1) at (-1,3.35) {};
        
    %\draw (C0) -- (C1);   
    \draw (C1) -- (M1);
    \draw (N1) -- (M2);
    %\draw (C1) -- (N1);

    \end{tikzpicture}};
        \node[xshift=210, yshift=210] (y3) {        \begin{tikzpicture}[scale=.6]
        \node [draw,fill,circle,inner sep = 0pt, minimum size = .25cm,label=below:{\small1}] (M1) at (-.5,0) {};
        \node [draw,fill,circle,inner sep = 0pt, minimum size = .25cm,label=below:{\small2}] (M2) at (.5,0) {};
        \node [draw,fill,circle,inner sep = 0pt, minimum size = .25cm,label=above:{\small$\ol{1}$}] (N1) at (0,2) {};
        
        %cone points
        %\node [draw,fill,circle,inner sep = 0pt, minimum size = .15cm,label=below:{\small0}] (C0) at (-1,-1.35) {};
        \node [draw,fill,circle,inner sep = 0pt, minimum size = .25cm,label=above:{\small$\ol{0}$}] (C1) at (-1,3.35) {};
        
    \draw (N1) -- (M2);   
    %\draw (C0) -- (M1);
    %\draw (C0) -- (M2);
    %\draw (C1) -- (N1);

    \end{tikzpicture}};
        
        \node[xshift=0, yshift=280] (x2y2) {        \begin{tikzpicture}[scale=.6]
        \node [draw,fill,circle,inner sep = 0pt, minimum size = .25cm,label=below:{\small1}] (M1) at (-.5,0) {};
        \node [draw,fill,circle,inner sep = 0pt, minimum size = .15cm,label=below:{\small2}] (M2) at (.5,0) {};
        \node [draw,fill,circle,inner sep = 0pt, minimum size = .25cm,label=above:{\small$\ol{1}$}] (N1) at (0,2) {};
        
        %cone points
        %\node [draw,fill,circle,inner sep = 0pt, minimum size = .15cm,label=below:{\small0}] (C0) at (-1,-1.35) {};
        \node [draw,fill,circle,inner sep = 0pt, minimum size = .25cm,label=above:{\small$\ol{0}$}] (C1) at (-1,3.35) {};
        
    \draw (C1) -- (M1);   
    %\draw (C0) -- (M1);
    \draw (N1) -- (M2);
    %\draw (C1) -- (N1);

    \end{tikzpicture}};
        \node[xshift=140, yshift=280] (xy3) {        \begin{tikzpicture}[scale=.6]
        \node [draw,fill,circle,inner sep = 0pt, minimum size = .15cm,label=below:{\small1}] (M1) at (-.5,0) {};
        \node [draw,fill,circle,inner sep = 0pt, minimum size = .25cm,label=below:{\small2}] (M2) at (.5,0) {};
        \node [draw,fill,circle,inner sep = 0pt, minimum size = .25cm,label=above:{\small$\ol{1}$}] (N1) at (0,2) {};
        
        %cone points
        %\node [draw,fill,circle,inner sep = 0pt, minimum size = .15cm,label=below:{\small0}] (C0) at (-1,-1.35) {};
        \node [draw,fill,circle,inner sep = 0pt, minimum size = .25cm,label=above:{\small$\ol{0}$}] (C1) at (-1,3.35) {};
        
    \draw (C1) -- (M1);   
    \draw (N1) -- (M2);
    %\draw (C0) -- (M2);
    %\draw (C1) -- (N1);

    \end{tikzpicture}};
        
        \draw[thick] (1) -- (x);
        \draw[thick] (1) -- (y);
        
        \draw[thick] (x) -- (x2);
        \draw[thick] (x) -- (xy);
        \draw[thick] (y) -- (xy);
        \draw[thick] (y) -- (y2);
        
        \draw[thick] (x2) -- (x2y);
        \draw[thick] (xy) -- (x2y);
        \draw[thick] (xy) -- (xy2);
        \draw[thick] (y2) -- (y3);
        \draw[thick] (xy2) -- (y2);
        
        \draw[thick] (x2y) -- (x2y2);
        \draw[thick] (xy2) -- (x2y2);
        \draw[thick] (xy2) -- (xy3);
        \draw[thick] (y3) -- (xy3);
    \end{tikzpicture}}

    \captionsetup{width=1.0\linewidth}
  \captionof{figure}{Bi-rooted forests of $\GG$}
 \label{fig:biroot}
 \end{center}
\end{figure}

\begin{figure}[H]
\renewcommand{\thefigure}{6C} 
 \begin{center}
 
  \scalebox{.75}{        \begin{tikzpicture}[scale=.6]
        \node (1) {\input{figures/wf1}};
        \node[xshift=70, yshift=70] (y) {\input{figures/wfy}};
        \node[xshift=-70, yshift=70] (x) {\input{figures/wfx}};
        
        \node[xshift=-140, yshift=140] (x2) {\input{figures/wfx2}};
        \node[xshift=0, yshift=140] (xy) {\input{figures/wfxy}};
        \node[xshift=140, yshift=140] (y2) {\input{figures/wfy2}};
        
        \node[xshift=-70, yshift=210] (x2y) {\input{figures/wfx2y}};
        \node[xshift=70, yshift=210] (xy2) {\input{figures/wfxy2}};
        \node[xshift=210, yshift=210] (y3) {\input{figures/wfy3}};
        
        \node[xshift=0, yshift=280] (x2y2) {\input{figures/wfx2y2}};
        \node[xshift=140, yshift=280] (xy3) {\input{figures/wfxy3}};
        
        \draw[thick] (1) -- (x);
        \draw[thick] (1) -- (y);
        
        \draw[thick] (x) -- (x2);
        \draw[thick] (x) -- (xy);
        \draw[thick] (y) -- (xy);
        \draw[thick] (y) -- (y2);
        
        \draw[thick] (x2) -- (x2y);
        \draw[thick] (xy) -- (x2y);
        \draw[thick] (xy) -- (xy2);
        \draw[thick] (y2) -- (y3);
        \draw[thick] (xy2) -- (y2);
        
        \draw[thick] (x2y) -- (x2y2);
        \draw[thick] (xy2) -- (x2y2);
        \draw[thick] (xy2) -- (xy3);
        \draw[thick] (y3) -- (xy3);
    \end{tikzpicture}}
    \captionsetup{width=1.0\linewidth}
  \captionof{figure}{Multicomplex of 2-weighted forests of $\GG$}
 \label{fig:multicomp}
 \end{center}
\end{figure}

%\begin{figure}[h]
%\renewcommand{\thefigure}{5D}
% \begin{center}
 
%    \scalebox{.75}{\input{figures/oi}}
%    \captionsetup{width=1.0\linewidth}
%  \captionof{figure}{Multicomplex}
 
% \end{center}
%\end{figure}

\section{Klee-Samper's strengthening of Stanley's conjecture}\label{sec:Klee-Samper}

Our approach to proving Stanley's conjecture for the graphic matroid of a biconed graph $G$ involved choosing an ordering on the edges of $G$ and singling out a `smallest' spanning tree ${\mathcal T}_0$.  We then used the remaining edges $E(G) \backslash {\mathcal T}_0$ as the variables for the multicomplex that realized the $h$-vector of $G$.  In \cite{KleeSamper2015} Klee and Samper propose a similar approach to proving Stanley's conjecture for all matroids, based on lexicographic shellability of their independence complexes.  For this we recall some definitions.

%This involves a map ${\mathcal F}$ from the family of \emph{based matroids} $(\Delta,B,<)$ to the family of pure order ideals on the variables $E(\Delta) \backslash B$.  In \cite[Conjecture 3.10]{KleeSamper2015} they posit the existence of such a map ${\mathcal F}$ that satisfies certain compatibility properties with respect to restriction.  They furthermore prove that their conjecture implies Stanley's conjecture.

\begin{definition}[\cite{KleeSamper2015}, Definition 3.9]

A \emph{based matroid} is a triple $(\Delta, B, <)$, where $\Delta$ is a matroid, $B$ is a basis of $\Delta$ and $<$ is a total order of $E(\Delta)-B$. For an independent set $I$ such that $I \cap B = \emptyset$, let $\Gamma_I$ be the matroid whose elements are the subsets $U$ of $B$ such that $U \cup I \in \Delta$. Two based matroids $(\Delta, B, <)$ and $(\Delta', B', <')$ are isomorphic if there is a matroid isomorphism $f:\Delta\rightarrow \Delta'$ such that $f(B)=B'$ and $f$ is order preserving on $E(\Delta)-B$.
\end{definition}

The notion of a based matroid is inspired by the fact that any ordering of a the ground set of a matroid $\Delta$ provides a lexicographic shelling of its independence complex. This leads to a decomposition of the $h$-vector of a matroid that suggests an inductive procedure to construct a pure multicomplex. Based on these notions, Klee and Samper formulate the following.

\begin{conjecture}[\cite{KleeSamper2015}, Conjecture 3.10]\label{con:KSconjecture}
Let $d$ be a fixed integer and let $\mathcal{A}^d$ be the family of based matroids of rank $d$. Then there exists a map $\mathcal{O}$ from $\mathcal{A}^d$ to the family of pure order ideals such that the following conditions hold for every based matroid $(\Delta, B, <)$.

\begin{enumerate}
\item The variables of $\mathcal{O}(\Delta, B, <)$ are $\{x_i| i \in E(\Delta) - B\}$.

\item Every monomial in $\mathcal{O}(\Delta, B, <)$ is supported on a set of the form $\{x_i| i \in I\}$ for some independent set $I$ of $\Delta $ with $I \cap B = \emptyset$.

\item For each independent set $I$ that is disjoint from $B$, there are exactly $h_j(\Gamma_I)$ monomials in $\mathcal{O}(\Delta, B, <)$ with degree $\abs{I}+j$ and support $\{x_i|i \in I\}$.

\item For each independent set $I$ that is disjoint from $B$, the restriction of $\mathcal{O}(\Delta, B, <)$ to the variables $\{x_i| i \in I\}$ is $\mathcal{O}(\Delta|_{B \cup I}, B, <)$.

\item If $(\Delta', B', <')$ is a based matroid and $f:(\Delta, B, <)\rightarrow (\Delta', B', <')$ is an isomorphism, then $\mathcal{O}(\Delta,B,<)$ is naturally isomorphic to $\mathcal{O}(\Delta',B',<')$ by relabeling the index of each variable in $(\Delta, B, <)$ with its image under $f$.
\end{enumerate}
\end{conjecture}

The authors of \cite{KleeSamper2015} show that Conjecture \ref{con:KSconjecture} implies Stanley's Conjecture for all matroids.  We will describe how our constructions lead to a proof of all parts of Conjecture \ref{con:KSconjecture} aside from Condition (5), for the class of graphic matroids of biconed graphs.

First recall that in our construction of biconed graphs, we have specified an order $<$ on the edge set and a lexicographically first spanning tree $T_0$.  Hence the graphic matroids coming from such graphs are naturally based matroids. In Corollary \ref{cor:mainbijection}, we defined a map $\phi$ from the spanning trees of a biconed graph $\G$ to the set $\F(\GG)$ of $2$-weighted forests of $\GG$, and in Lemma \ref{lem:pure} we saw that $\F(\GG)$ is a pure ordeal ideal.  Hence we can think of $\phi$ as a map from the set of based graphic matroids coming from biconed graphs to the set of pure order ideals.   With this in place we can prove the following.

\begin{theorem}\label{thm:KleeSamper}
Let ${\mathcal B}$ denote the set of based matroids arising from biconed graphs as described above.  Then $\phi$, thought of as a map from ${\mathcal B}$ to the family of pure ordeal ideals, satisfies conditions $(1)-(4)$ listed in Conjecture \ref{con:KSconjecture}.
\end{theorem}

\begin{proof}
Let $(\Delta, T_0, <)$ be the based graphic matroid associated to a biconed graph $\G$. By construction, the variables of $\phi(\Delta, T_0, <)$ are the edges of $\G$ that are not included in the lex-first tree $T_0$.  In addition, every monomial in $\phi(\Delta, T_0, <)$ corresponds to a multiset of $T \backslash T_0$, where $T$ is the set of edges of some spanning tree of $\G$. Hence the first two items of Conjecture \ref{con:KSconjecture} are satisfied.

For item $(3)$, suppose $I$ is an independent set in $\G$ that is disjoint from $T_0$.  Then $\Gamma_I$ is the graphic matroid of the graph obtained by restricting to the edges $T_0 \cup I$ and then contracting the edges of $I$.  This correspondence provides a bijection between 1) the spanning trees of $\G$ which contain $I$ and which are contained in $T_0 \cup I$, and 2) the spanning trees of $\Gamma_I$ .  It remains to show that this map decreases passivity  by $\abs{I}$, since $h_j(\Gamma_I)$ is the number of spanning trees of $\Gamma_I$ with passivity $j$. Since $I$ is disjoint from $T_0$, we have from Lemma \ref{lem:passive} that each edge contracted under the map is internally passive in $\G$. The remaining edges are passive exactly when they can be replaced by a smaller edge to again obtain a spanning tree.  However all edges in $I$ are larger than all edges of $T_0$, so contracting or deleting edges which are not contained in $T_0$ does not affect passivity of the edges contained in $T_0$. Thus, the number of passive edges drops by $\abs{I}$ under this map.

For item $(4)$, we want to show that for an independent set $I$ disjoint from $T_0$, the multicomplex obtained from $\F(\GG)$ by restricting to the variables $x_i$ for $i \in I$ is the same as the multicomplex obtained by applying $\phi$ to the graph $\G|_{T_0 \cup I}$, the restriction of $\G$ to $T_0 \cup I$. In other words, deleting edges of $\G$ not contained in $T_0 \cup I$ doesn't affect the monomials whose support is contained in $I$. This follows since the monomial associated with a spanning tree in $\G$ can be computed using only knowledge of $T_0$ and the spanning tree in question. 
%For a spanning tree $T$ of $\G$, the monomial associated with it can be computed by only looking at the connected components in $T \setminus T_0$ along with knowledge of the connecting vertices.
In particular it does not depend on any edges in $E(\Delta) \setminus B$, so deleting these edges in the underlying matroid still gives the same monomial for each independent set common to both matroids.
\end{proof}
%The matroid $(\Delta|_{B \cup I}, B,<)$ is the graphic matroid corresponding to the graph with edges not in $I$ nor $B$ deleted thus $\F(\Delta|_{B \cup I}, B,<))$ has determinants given by the edges of $I$ as desired.

%Preston: I've updated item (5).

%Finally, item $(5)$ follows since our construction of a monomial associated to a spanning tree is determined by the ordering of the edges of $E(G) \backslash T_0$. Hence if $f$ is an isomorphism of based matroids of biconed graphs 

We have thus far been unable to establish Condition (5) of Conjecture \ref{con:KSconjecture} for our class of matroids, as it seems difficult to determine conditions under which two based matroids coming from different biconed graphs are isomorphic. On the other hand, it is not clear if a weaker statement is sufficient in order to apply the constructions to Stanley's Conjecture.

As is explained in \cite{HeaSam}, if graphic matroids of biconed graphs did indeed satisfy all conditions of Conjecture \ref{con:KSconjecture} then this would imply that the poset of divisibility of the pure order ideal associated to a biconed graph $\G$ is an extension of a certain partial order on bases defined by Las Vergnas in \cite{Las}.  To recall this notion, suppose $<$ is an ordering of the ground set of a matroid ${\mathcal M}$. We define a partial order on the bases of ${\mathcal M}$ given by inclusion of its internally passive elements, and call this the \emph{internal order} of the (ordered) matroid ${\mathcal M}$.  One can check that the poset depcited in Figure $6$A is in fact the internal order of the matroid with the given ordering on the edges of the graph.

For the case of $K_4$ (which is an example of a biconed graph), it is known that the internal order of the underlying graphic matroid is not the face poset of any multicomplex (see for instance \cite{Dal}).  Hence the poset of divisibility corresponding to the multicomplex $\F(\red)$  that we create for $K_4$ will strictly contain the poset associated to the internal order.

\section{Open questions and future directions}\label{sec:Future}

In this section, we discuss questions that arise from our study of biconed graphs which are possible directions for future research. 

\begin{question}
What is the M\"obius coinvariant $\mu^\perp(K_{n_1, n_2, \dots, n_\ell})$ of a complete multipartite graph?
\end{question}

Recall that the M\"obius conivariant $\mu^\perp(G)$ of a graph $G$ is defined to be $|\mu_{L({\mathcal M}(G))^*}|$, the M\"obius invariant of the lattice of flats of the matroid dual to the graphic matroid of $G$.  It is known that $\mu^\perp(G)$ is equal to the rank of the reduced homology of the independence complex of ${\mathcal M}(G)$ and also equal to the Tutte evaluation $T_G(0,1)$, and hence equal to the last nonzero entry in the $h$-vector of the underlying matroid (counting the number of spanning trees with zero internal activity). In \cite{KookLee2018}, the authors found closed formulas for the M\"{o}bius coinvariants of complete bipartite graphs by counting certain edge and vertex rooted forests using Hermite polynomials. From our results, we see that determining the M\"{o}bius coinvariant of a biconed graph is equivalent to counting its maximal $2$-weighted forests, those which have in every non-singular component at least $1$-edge root and exactly one component with excess weight $2$. A careful count of such structures would then lead to a combinatorial formula the M\"{o}bius coinvariants of these graphs. A potential method is to take an approach similar to that in \cite{KookLee2018}, creating a structure with a blend of edge weighting and vertex rooting, in order to count the number of maximal $2$-weighted forests.

\begin{question} \label{Ferrers}
Can we use our characterization of the $h$-vectors of biconed graphs to get a nice formula for the case of Ferrers graphs?
\end{question}

Recall from \cite{EhrenborgvanWilligenburg2004} that a \emph{Ferrers graph} is a bipartite graph with vertex set partition $U = \{u_0, \dots, u_n\}$ and $V = \{v_0, \dots, v_m\}$ satisfying
\begin{itemize}
    \item If $(u_i,v_j)$ is an edge then so is $(u_p,v_q)$ for all $0 \leq p \leq i$ and $0 \leq q \leq j$,
    
    \item $(u_0, v_m)$ and $(u_n, v_0)$ are both edges.
\end{itemize}

In particular such a graph is biconed. For a Ferrers graph $G$ we have the associated partition $\lambda = (\lambda_0, \lambda_1, \dots, \lambda_n)$, where $\lambda_i$ is the degree of the vertex $u_i$.  The associated Ferrers diagram (also called Young diagram) is the diagram of boxes where we have a box in position $(i, j)$ if and only if $(u_i
, v_j)$ is an edge in $G$. Ehrenborg and van Willigenburg studied enumerative aspects of Ferrers graphs in \cite{EhrenborgvanWilligenburg2004}.  In \cite{DukeSeligSmithSteingrimsson2019} the authors studied (minimal) recurrent configurations of Ferrers graphs using decorated EW-tableaux. As explained for example in \cite{CorryPerkinson2018}, recurrent configurations of any graph $G$ are in a simple duality with superstable configurations of $G$, which by results of Merino \cite{Merino2001} form a multicomplex whose $f$-vector recovers the $h$-vector of the dual matroid ${\mathcal M}(G)^*$.  In particular the number of recurrent configurations of $G$ of a particular degree are given by the $h$-vector of ${\mathcal M}(G)^*$.  Recurrent and superstable configurations are objects of chip-firing, for more details we refer to \cite{CorryPerkinson2018}. Our work provide an interpretation for the $h$-vector of the primal matroid ${\mathcal M}(G)$, is there a relationship between our constructions and chip-firing? Also, Ferrers graphs can be obtained by taking the biconing of other Ferrers graphs, so is there any recursive structure that can be taken advantage of when investigating $h$-vectors?

\begin{question}
Does the set of $2$-edge rooted forests in a biconed graph lead to a basis for the homology of the matroid independence complex?

\end{question}

For any graph $G$ the independence complex of its graphic matroid $I({\mathcal M}(G))$ is a wedge of spheres of dimension $n-2$ (where $n$ is the number of vertices).  The number of spheres in this wedge is given by the M\"obius coninvariant $\mu^\perp(G)$. In the case of biconed graphs, we have seen that this number is given by the number of maximal $2$-edge rooted forests. Hence a natural question to ask if one can associate a $2$-edge rooted forest with a fundamental cycle in $\tilde H(I({\mathcal M}(G)), {\mathbb R})$, to obtain a basis for this vector space.  Furthermore, the automorphism group $\text{Aut}(G)$ of the graph $G$ acts on this vector space, and one can perhaps use such a basis to study this representation.  This was worked out for the case of coned graphs in \cite{Kook2007hom} and for complete bipartite graphs in \cite{KookLee2018}.

\begin{question} \label{n-coned}
Can we generalize our constructions to $n$-coned graphs?
\end{question}

Generalizing the construction of biconed graphs we can construct \emph{$n$-coned graphs} $G^{U_1,U_2,\dots U_n}=(V(G),E(G))$ by taking a graph and connecting each of its vertices to at least one of $n$ coning points. The internal activity of edges in $n$-coned graphs seems to parallel that of biconed graphs, so $k$-weighted forests or another weighted forest structure may lead to pure multicomplexes. Does our multicomplex structure extend in a natural way to more cone vertices? What requirements do we need on the edges between the coning points so that we obtain a pure multicomplex?

\begin{question} \label{radius 2}
Does Stanley's Conjecture hold in the case of matroids of radius $2$ graphs?
\end{question}
%PRESTON: Reworded things here.
The \emph{eccentricity} of a vertex $v$ in a connected graph $G$ is the number of edges between it and the vertex farthest (with respect to edges) from it. The \emph{radius} of $G$ is the minimum eccentricity of its vertices. Biconed graphs are special cases of radius $2$ graphs since the eccentricity of both coning vertices is at most $2$. This generalizes the class of coned graphs, which are exactly the graphs of radius $1$. Stanley's Conjecture has been easiest to prove for graphs which are well-connected -- roughly speaking, graphs whose ratio of cardinality of edges to cardinality of vertices is high -- so radius $2$ graphs may be a reasonable next step.

%EVAN: Added this question about enumerating the spanning trees of a bipartite graph. Please edit and add as much as you see fit.
\begin{question} \label{span trees}
Can we bound the number of spanning trees of a biconed bipartite graph?
\end{question}
Let $G$ be a bipartite graph with $m$ vertices on one side and $n$ vertices on the other side, with vertex degrees $d_1, d_2, \dots, d_m$ and $e_1, e_2, \dots, e_n$. Is it true that the number of spanning trees of $G$ is at most
$$
m^{-1}n^{-1}\prod_{i=1}^m d_i \prod_{j=1}^n e_j ?
$$
Ehrenborg and van Willigenburg proved this for Ferrer's graphs, where in fact equality is achieved \cite{EhrenborgvanWilligenburg2004}. Klee and Stamps give a linear algebraic approach for weighted graphs using the Weighted Matrix-Tree Theorem in \cite{KleeStampsWeight2019}. They use a similar linear algebraic approach for unweighted graphs using Lapacian matrices and Kirchhoff's Matrix-Tree Theorem in \cite{KleeStampsUnweight2019}.

\section*{Acknowledgements}

This research was primarily conducted under NSF-REU grant DMS-1757233 during the Summer 2019 Mathematics REU at Texas State University. The authors gratefully acknowledge the financial support of NSF and also thank Texas State for providing support and a great working environment. We also thank the anonymous referees for helpful comments that have lead to substantial improvements, and in particular for pointing out the connection to the work of Klee and Samper.  This also lead to some streamlining of our arguments.

%%%%%%%%%%%%%%%%%%%%%%%%%%%%%%%%%%%%%%%%%%%%%%%%%%%%%%%%%%%%%%%%%%%%
\bibliography{Biconed_final}
\bibliographystyle{siam}
%\bibliographystyle{abbrv}
%%%%%%%%%%%%%%%%%%%%%%%%%%%%%%%%%%%%%%%%%%%%%%%%%%%%%%%%%%%%%%%%%%%%

\end{document}